\documentclass[a4paper,10pt]{amsart}

\usepackage{enumerate}
\usepackage[latin1]{inputenc}
\usepackage{amsmath}
\usepackage{amsthm}
\usepackage{latexsym}
\usepackage{amssymb}
\usepackage{mathtools}
\mathtoolsset{showonlyrefs}
\usepackage[colorlinks=true,urlcolor=blue,
citecolor=red,linkcolor=blue,linktocpage,pdfpagelabels,
bookmarksnumbered,bookmarksopen]{hyperref}
\usepackage[left=2.61cm,right=2.61cm,top=2.72cm,bottom=2.72cm]{geometry}
\usepackage[hyperpageref]{backref}
\usepackage[all]{xy}
\usepackage{calc}
\usepackage{multicol}
\newtheorem{thm}{Theorem}[section]
\newtheorem{prop}[thm]{Proposition}
\newtheorem{definition}[thm]{Definition}
\newtheorem{lem}[thm]{Lemma}
\newtheorem{cor}[thm]{Corollary}

\newtheorem{remark}[thm]{Remark}
\theoremstyle{notation}
\newtheorem*{notation}{Notation}

\newcommand{\R}{\mathbb{R}}
\numberwithin{equation}{section}
\newcommand{\N}{\mathbb{N}}

\newcommand{\eps}{\varepsilon}

\begin{document}

\title[Normalized solutions for 4NLS]{Normalized solutions to the mixed dispersion nonlinear Schr\"odinger equation in the mass critical and supercritical regime}

\thanks{D. Bonheure \& J.B. Casteras are supported by MIS F.4508.14 (FNRS), PDR T.1110.14F (FNRS); J.B. Casteras is supported by the
Belgian Fonds de la Recherche Scientifique -- FNRS;
D. Bonheure is partially supported by the project ERC Advanced Grant  2013 n. 339958: ``Complex Patterns for Strongly Interacting Dynamical Systems - COMPAT'' and by ARC AUWB-2012-12/17-ULB1- IAPAS, This work has been carried out in the framework of the project NONLOCAL (ANR-14-CE25-0013),
funded by the French National Research Agency (ANR)}

\author[Bonheure, Casteras, Gou and Jeanjean]{Denis Bonheure \and Jean-Baptiste Casteras \and Tianxiang Gou \and Louis Jeanjean}

\address{Denis Bonheure, Jean-Baptiste Casteras
\newline \indent D\'epartement de Math\'ematiques, Universit\'e Libre de Bruxelles,
\newline \indent CP 214, Boulevard du triomphe, B-1050 Bruxelles, Belgium}
\email{Denis.Bonheure@ulb.ac.be}
\email{jeanbaptiste.casteras@gmail.com}
\address{Tianxiang Gou
\newline \indent Laboratoire de Math\' ematiques (UMR 6623), Universit\' e Bourgogne Franche-Comt\'e,
\newline \indent 16, Route de Gray 25030 Besan\c con Cedex, France.
\newline \indent School of Mathematics and Statistics,
\newline \indent Lanzhou University, Lanzhou, Gansu 730000, People's Republic of China.}
\email{gou.tianxiang@gmail.com}
\address{Louis Jeanjean
\newline \indent Laboratoire de Math\' ematiques (UMR 6623), Universit\' e Bourgogne Franche-Comt\'e,
\newline \indent 16, Route de Gray 25030 Besan\c con Cedex, France.}
\email{louis.jeanjean@univ-fcomte.fr}

\begin{abstract}
In this paper, we study the existence of solutions to the mixed dispersion nonlinear Schr\"odinger equation
$$
\gamma \Delta ^2 u -\Delta u + \alpha u=|u|^{2 \sigma} u, \quad u \in H^2(\R^N),
$$
under the constraint
$$
\int_{\R^N}|u|^2 \, dx =c>0.
$$
We assume $\gamma >0, N \geq 1, 4 \leq \sigma N < \frac{4N}{(N-4)^+}$, whereas the parameter $\alpha \in \R$ will appear as  a Lagrange multiplier. Given $c \in \R^+$, we consider several questions including the existence of ground states, of positive solutions and the multiplicity of radial solutions. We also discuss the stability of the standing waves of the associated dispersive equation. 
\end{abstract}

\maketitle
\tableofcontents
\section{Introduction.}
In this paper we consider the biharmonic NLS (Nonlinear Schr\"odinger Equation) with mixed dispersion
\begin{equation}
\label{4nlsdis}
i \partial_t \psi -\gamma \Delta^2 \psi +\Delta \psi +|\psi|^{2\sigma} \psi =0, \ \psi (0, x)=\psi_0 (x),\ (t, x) \in \R \times \R^N,
\end{equation}
where we always assume that $\gamma >0$ and $0 < N \sigma < 4^*$. Here 
$$4^* := \frac{4N}{(N-4)^+}, \mbox{ namely } 4^*= \infty \mbox{ if } N\leq 4, \mbox{ and  }4^* = \frac{4N}{N-4} \mbox{ if } N \geq 5.$$  For $1 \leq q < \infty,$ we denote by $L^q(\R^N)$  the usual Lebesgue space with norm $||u||_q^q:= \int_{\R^N}|u|^q \, dx. $ 

\medbreak

Let us recall some basic facts on \eqref{4nlsdis} illustrating the effect of $\gamma>0$. First, it is well-known that NLS i.e. \eqref{4nlsdis} with $\gamma=0$, is locally well-posed in $H^1 (\R^N)$ provided that $\sigma N <2^\ast=\frac{2N}{N-2}$ (see \cite[Theorem 4.3.4]{Ca}). The counter-part of this result for \eqref{4nlsdis} with $\gamma>0$ has been obtained by Pausader \cite[Proposition 4.1]{Pa} namely he proved that \eqref{4nlsdis} is locally well-posed in $H^2 (\R^N)$ for $\sigma N < 4^\ast$. Concerning the global well-posedness, it is well-known that it holds in $H^1 (\R^N)$ for NLS when $N\sigma < 2/N$ (see \cite[Corollary 6.1.2]{Ca}) and in $H^2 (\R^N)$ for \eqref{4nlsdis} when $N\sigma < 4/N$ (see \cite[Theorem 1.1]{Pa} for radially symmetric initial data but this assumption can be released as claimed in \cite{BoLe}). Moreover, in this regime i.e. $N\sigma < 2/N$, ground state solution for NLS are orbitally stable (see \cite[Theorem 8.3.1]{Ca}). On the contrary, it is well known that, when $\sigma N \geq 2$, NLS can become singular at finite time, see for instance \cite[Theorem 6.5.10]{Ca}.  Karpman and Shagalov \cite{KaSh}, see also \cite{Ka},  were apparently the first to study the regularization and stabilization effect of a small fourth-order dispersion. Using a combination of stability analysis and numerical simulations, they showed that standing wave solutions are orbitally stable for any $\gamma >0$, when $0<\sigma N \leq 2$, and when $ 2<\sigma N <4$ for $\gamma >0$ small enough. When $\sigma N >4$ they observed an instability phenomenon. Thus  $\sigma N=4$ appears as a new critical exponent and adding a small fourth-order dispersion term clearly  helps to stabilize the standing waves. 

In nonlinear optics, NLS is usually derived from NLH (nonlinear Helmhotz equation) for the electric field by separating the fast oscillations from the slowly varying amplitude. In the so-called paraxial approximation, NLS appears in the limit as the equation solved by the dimensionless electric-field amplitude, see e.g. \cite[Section 2]{FiIlPa}. The fact that its solutions may blow up in finite time suggests that some small terms neglected in the paraxial approximation play an important role to prevent this phenomenon. Therefore a small fourth-order dispersion term was proposed in \cite{FiIlPa}, see also  \cite{BaFi, BaFiMa1, BaFiMa}, as a nonparaxial correction, which eventually gives rise to \eqref{4nlsdis}. In particular applying some arguments developed in \cite{We}, the authors \cite{FiIlPa} proved that all solutions to \eqref{4nlsdis} exist globally in time for $\sigma N <4$. We mention that the necessary Strichartz estimates have been previously obtained by Ben-Artzi et al \cite{BeKoSa}.

Nevertheless, despite its physical relevance, the dispersion equation \eqref{4nlsdis} is far from being well understood. There are only few papers dealing with \eqref{4nlsdis} besides the ones already mentioned \cite{BoCaGoJe, BoCaMoNa, BoLe, Ka, NaPa, Pa, PaSh, PaXi}. \medskip

In this paper we are interested in standing waves solutions, that is solutions of the form $\psi (t,x)=e^{i\alpha t} u(x)$, $\alpha >0$. The function $u$ then satisfies the elliptic equation
\begin{equation}
\label{4nls}
\gamma \Delta^2 u - \Delta u+\alpha u=|u|^{2\sigma}u, \quad u \in H^2(\R^N).
\end{equation}
A possible choice is to consider that $\alpha >0$ is given and to look for solutions $u \in H^2(\R^N)$ of \eqref{4nls}. Such solutions correspond to critical points of the functional 
\begin{equation}
\label{J-quadr}
I(u):=\frac{\gamma}{2}\int_{\R^N} |\Delta u|^2\, dx+ \frac 12 \int_{\R^N}|\nabla u|^2\, dx+\frac{\alpha}{2} \int_{\R^N}|u|^2\, dx - \dfrac{1}{2\sigma +2}\int_{\R^N} | u|^{2\sigma +2}\, dx,
\end{equation}
and of particular interest are the so-called least energy solutions. Namely solutions which minimize $I$ on the set
\begin{equation}
\label{ground-state}
\mathcal{N}:= \{u \in H^2(\R^N) \backslash \{0 \} : I'(u) =0 \}.
\end{equation}
This point of view is adopted in the paper \cite{BoNa}, see also \cite{BoCaMoNa}. \medskip

Alternatively one can consider the existence of solutions to \eqref{4nls} having a prescribed $L^2$-norm. Since solutions  $\psi \in C([0, T); H^2(\R^N))$ to \eqref{4nlsdis} conserve their {\it mass} along time, i.e. $\|\psi(t)\|_2=\|\psi(0)\|_2$ for $t \in [0, T)$, it is natural, from a physical point view, to search for such solutions. 

Here we focus on this issue. For $c>0$ given, we consider the problem of finding solutions to 
\begin{equation}\label{Pc}\tag{$P_c$} \gamma \Delta^2 u - \Delta u +  \alpha u = |u|^{2 \sigma}u  \quad \mbox{ with } \quad  \int_{\R^N}|u|^2 dx = c.
\end{equation}
 
It is standard to show that a critical point of the energy functional, 
\begin{equation}\label{def:E}
E (u):=\frac{\gamma}{2}\int_{\R^N}|\Delta u|^2\, dx+\frac{1}{2}\int_{\R^N}|\nabla u|^2\, dx-\frac{1}{2\sigma+2}\int_{\R^N}|u|^{2\sigma+2}\, dx.
\end{equation}
restricted to
\begin{equation*}\label{de:Mmu}
S(c):=\{u\in H^2(\R^N):\int_{\R^N}|u|^{2}\, dx=c\}
\end{equation*}
corresponds to a solution of \eqref{Pc}. The value of $\alpha \in \R$ in \eqref{Pc} is then an unknown of the problem and it corresponds to the associated Lagrange parameter.  \medskip

In \cite{BoCaMoNa}, the authors study, for $c>0$ given, the minimizing problem
\begin{equation}\label{MinL2fixed}
m(c):=\inf_{u\in S (c) }E(u).
\end{equation}
When $0 < \sigma N <4,$ the functional $E$ is bounded from below on $S(c)$ for any $c>0$, which makes possible to search for a critical point of $E$ restricted to $S(c)$ as a global minimizer. In that direction the following  result was obtained in \cite{BoCaMoNa}.
\begin{thm}\label{Compact-Min-Sol}
Assume  that $0<\sigma N <2$, then $m(c)$ is achieved for every $c>0$. If $2 \leq \sigma N <4$  then there exists a critical mass $\tilde{c}= \tilde{c}(\sigma, N)$ such that 
\begin{enumerate}[(i)]
\item $ m(c)$ is not achieved if $c< \tilde c$;
\item $ m(c)$ is achieved if $c > \tilde c$ and $\sigma = 2/N$;
\item $ m(c)$ is achieved if $c \ge \tilde c$ and $\sigma \neq 2/N$.
\end{enumerate}
Moreover if $\sigma$ is an integer and $m(c)$ is achieved, then there exists at least one radially symmetric minimizer.
\end{thm}

We note the appearance of a critical mass when $2 \leq \sigma N < 4$. It is linked to the fact that each three terms of $E$ behaves differently with respect to dilations. Such a phenomenon was first observed in \cite{CoJeSq}, see also \cite{CaDoSaSo,JeLu} for related results. \medskip

In this paper we focus on the mass-critical case $\sigma N=4$ and on the mass-supercritical case $ 4 < \sigma N < 4^*.$ Our first result concerns the case $\sigma N=4$.  A key role is played by the Gagliardo-Nirenberg inequality (see e.g. \cite[Theorem in Lecture II]{Nir})
\begin{equation}\label{G-N-H2-ineq}
 \|u\|^{2\sigma+2}_{2\sigma+2}\leq B_N(\sigma)\|\Delta u\|^{\frac{\sigma N}{2}}_2\|u\|^{2+2\sigma-\frac{\sigma N}{2}}_2,
\end{equation}
where
\begin{align*}
     \begin{cases}
     &0\leq\sigma, \hspace{.5cm} \mbox{if}\ N\leq 4,\\
     &0\leq\sigma<\dfrac{4}{N-4}, \hspace{.5cm} \mbox{if}\ N\geq 5,
     \end{cases}
\end{align*}
and $B_N(\sigma)$ is a constant depending on $\sigma$ and $N$. 
\begin{thm} \label{thm1}
Let $N \geq 1, \sigma N=4$. There exists a $c^{*}_N >0$ such that
\[ m(c):= \inf_{u \in S(c)}E(u)=
\left\{
\begin{aligned}
&0,  &0<c\leq c^*_N,\\
&-\infty, &c> c^*_N.
\end{aligned}
\right.
\]
For $c \in (0, c_N^*],$ \eqref{Pc} has no nontrivial solution and in particular $m(c)$ is not achieved. In addition, we have $c^{*}_N= (\gamma C(N) )^{\frac N4}$ where 
\begin{align} \label{cn}
C(N):=\frac{N+4}{N B_N (\frac 4N) },
\end{align}
and $B_N (\sigma)$ is the smallest constant satisfying  \eqref{G-N-H2-ineq}.
\end{thm}
Theorem \ref{thm1} shows that $m(c)$ becomes infinite for $c>0$ large. Actually when $4 < \sigma N <4^*$ it is also the case for any $c>0$. To see this, for any 
 $u \in S(c)$, $\lambda>0$, we define
\begin{align} \label{dilatation}
u_{\lambda}(x):=\lambda^{\frac{N}{4}}u(\sqrt{\lambda} x).
\end{align}

This definition is clearly motivated by the fact that $\|u_\lambda\|_2=\|u\|_2$. One then computes
\begin{align} \label{functional}
E(u_{\lambda})=\frac{\gamma\lambda^{2}}{2}\int_{\R^N}|\Delta u|^2\, dx+\frac{\lambda}{2}\int_{\R^N}|\nabla u|^2\, dx-\frac{\lambda^{\sigma N/2}}{2\sigma+2}\int_{\R^N}|u|^{2\sigma+2}\, dx.
\end{align}
Obviously $E(u_{\lambda}) \to -\infty$ as $\lambda \to \infty$, and therefore $m(c)=-\infty$. In particular it is no more possible to obtain a critical point of $E$ restricted to $S(c)$ as a global minimizer. To overcome this difficulty, we exploit the property that $E$ restricted to $S(c)$ possesses a natural constraint, namely a set, that contains all the critical points of $E$ restricted to $S(c)$. This set is given by
$$
\mathcal{M}(c):=\{u\in S(c) : Q(u)=0\},
$$
where\
\begin{equation*}
\label{defiQ}
Q(u):=\gamma  \int_{\R^N}|\Delta u|^2\, dx+\dfrac{1}{2}\int_{\R^N}|\nabla u|^2\, dx-\frac{\sigma N}{2(2\sigma+2)}\int_{\R^N}|u|^{2\sigma+2}\, dx .
\end{equation*}
Using \eqref{functional}, we see that
\begin{equation}\label{derivationofQ}
Q(u)=\dfrac{\partial E (u_\lambda)}{\partial \lambda}|_{\lambda=1}
\end{equation}
and thus, at least heuristically, the set $\mathcal{M}(c)$ contains all the critical points. Actually the condition $Q(u)=0$ corresponds to a Derrick-Pohozaev identity. As we shall see,  the functional $E|_{\mathcal{M}(c)}$ has much better properties than $E|_{S(c)}$. In particular, it is coercive, see Lemma \ref{coercive}. \medskip

For convenience, we define $c_0 \in \R$ as $c_0 =0$ if $4 < \sigma N <4^*$ and $c_0 = c_N^*$ if $\sigma N =4$. 
In Lemma \ref{Mnonvoid}, we shall prove that $\mathcal{M}(c) \neq \emptyset$, for any $c >c_0$. It is then possible to define, for any fixed $c > c_0$
\begin{align}\label{min1}
\Gamma(c):= \inf_{u\in \mathcal M(c)}E(u).
\end{align}
Our next result concerns the existence of a minimizer associated to $\Gamma(c)$. Note that, when a minimizer exists, it is a ground state solution in the sense that it minimizes the functional $E$, among all solutions having this $L^2$-norm. \medskip

\begin{thm}
\label{thmmain}
Let $N \geq 1$ and $4 \leq \sigma N< 4^*$.  There exists $c_{N,\sigma}>c_0$ such that for any $c\in (c_0,c_{N,\sigma})$,
\eqref{Pc} has a  ground state solution $u_c$, satisfying $E(u_c)=\Gamma(c)$, and the associated Lagrange parameter $\alpha_c$ is strictly positive. Moreover
\begin{enumerate}[(i)]
\item if $N=1,2$ and $\sigma \geq 4/N$, then $c_{N,\sigma }=\infty$;
\item if $4/3 \le\sigma \leq 2$, then $c_{3,\sigma}=\infty$; 
\item if $\sigma N=4$, then $c_{4,1}=\infty$ and $c_{N, \sigma} \geq  \left(\frac{N}{N-4}\right)^{\frac N4}c^*_N$ for $N \geq 5$.
\end{enumerate}
\end{thm}

The proof of Theorem \ref{thmmain} crucially relies on the observation, strongly based on arguments due to Bartsch and Soave \cite{BaSo2}, that there exists a Palais-Smale sequence, for $E$ restricted to $S (c)$ at the level $\Gamma(c)$, which consists of elements of $\mathcal{M}(c)$. \medskip

Let $(u_n)_n \subset \mathcal{M}(c)$ be a Palais-Smale sequence for $E$ restricted to $S (c)$.  Since $E$ is coercive on $\mathcal{M}(c)$, see Lemma \ref{coercive}, we infer that $(u_n)_n$ is bounded and it follows that, up to a subsequence and up to translations,  $u_n \rightharpoonup u_c$ weakly in $H^2(\R^N)$ for some $u_c$ satisfying
\begin{equation} \label{weaklimite}
\gamma \Delta^2 u_c - \Delta u_c + \alpha_c u_c = |u_c|^{2 \sigma}u_c.
\end{equation}

We show in Lemma \ref{conv} that the strong convergence in $H^2(\R^N)$ occurs as soon as $(u_n)_n$ strongly converges in $L^{2 \sigma +2}(\R^N)$ and  $\alpha_c >0$ in \eqref{weaklimite}.   \medskip

In the frame of Theorem \ref{thmmain}, since we work with a Palais-Smale sequence at the level $\Gamma(c)$, we can make use of the observation that $c \mapsto \Gamma(c)$ is nonincreasing on $(c_0, \infty)$ to prove that $u_n \to u_c$ in $L^{2 \sigma +2}(\R^N)$,  see Lemma \ref{propcompactness}. The restrictions on the couple $(N, \sigma)$ arise to insure that $\alpha_c >0$, see Lemma \ref{sign-la}. Note that the strong convergence of $(u_n)_n$ to $u_c$ in $H^2(\R^N)$ implies that  $u_c$  belongs to $\mathcal{M}(c)$ and that it satisfies $E(u_c) = \Gamma(c)$. In particular $\Gamma(c)$ is then achieved. \medskip

Finding constrained critical points when the functional is unbounded from below on the constraint is a question that remained for long only considered in the sole paper \cite{Je}. This question however has recently received more attention and in particular in the frame of various scalar problems \cite{AcWe,BaVa, BeJe, BeJeLu} as well as for systems \cite{BaJe, BaJeSo, BaSo, BaSo2}. The present work benefits in particular from techniques introduced in \cite{BaSo2, BeJeLu}. The common feature of these papers is that the underlying problems are autonomous and set on the whole space $\R^N$. This transfers to the functional a regular behaviour with respect to dilations that are essential in our proofs. In related works \cite{NoTaVe1,NoTaVe2,PiVe} where the underlying equations, or systems, are set on a bounded domain and are not necessarily autonomous, the questions tackled and the techniques used are quite different. 

In our next result, taking advantage of the genus theory, we prove the existence of infinitely many radial solutions to \eqref{Pc}. 
\begin{thm} \label{thm}
Assume that $N \geq 2$.
\begin{enumerate}[(i)]
\item If $4 < \sigma N < 4^*$, then for any $c\in (0, {c}_{N,\sigma})$, where ${c}_{N,\sigma}$ is defined in Theorem \ref{thmmain}, \eqref{Pc} possesses infinitely many radial solutions.
\item If $\sigma N=4$ and $2 \leq N \leq 4$, then for any $k \in \N^+$ there exists a $c(k) > c_N^*$ such that, for any $c \geq c(k),$ \eqref{Pc} possesses at least $k$ radial solutions. 
\end{enumerate}
\end{thm}

To establish Theorem \ref{thm} we work in the subspace $H_{rad}^2(\R^N)$ of radially symmetric functions in $H^2(\R^N)$. Accordingly we define 
$\mathcal{M}_{rad}(c) := \mathcal{M}(c) \cap H_{rad}^2(\R^N)$.  The proof of Theorem \ref{thm} is based on the use of the Kranosel'skii genus and again on arguments from \cite{BaSo2}, which guarantee, at appropriate minimax levels, the existence of a Palais-Smale sequence, for $E$ restricted to $S(c)$, consisting of elements of $\mathcal{M}_{rad}(c)$. To show the convergence of these Palais-Smale sequences we proceed as for the proof of Theorem \ref{thmmain}.
The required compactness comes here from
the compact embedding of $H_{rad}^2 (\R^N)$ into $L^{2 \sigma +2}(\R^N)$ whereas the positivity of the associated $\alpha_c$ in the limit equation \eqref{weaklimite} comes again from 
Lemma \ref{sign-la}. 
Another step in the proof of Theorem \ref{thm} is to show that the set $\mathcal{M}_{rad}(c)$  is {\it sufficiently large}. This is always the case when $4 < \sigma N <4^*$ for any $c>0$. However when $\sigma N=4$ the set $\mathcal{M}_{rad}(c)$ may be {\it too small}. In particular it shrinks to the empty set as $c \to c_N^*$. To obtain a given number of critical points we require that $c >c_N^*$ is sufficiently large. 

\medskip

In dimensions $3$ and $4$, and for radial ground state, we can relax the range assumption on $c$ of Theorem \ref{thm}.  By a radial ground state, we mean a radial function which minimizes $E$ on the set $ \mathcal{M}_{rad}(c)$.
\begin{thm}
\label{radN34} 
Let $N=3$ or $N=4$ and assume that $4\le\sigma N < \infty$. Then, for any $c>c_0$, there exists a radial ground state solution to \eqref{Pc}.
\end{thm}
The main additional ingredient in the proof is a sharp decay estimate of radial solutions $u$ to 
\begin{equation}
\label{alphazerointro}
\gamma\Delta^2 u -\Delta u =|u|^{2\sigma}u, 
\end{equation}
when $N=3$ and $\sigma> 2$ or $N=4$ and $\sigma>1$.
We believe this result has its own interest. 

\begin{prop}
\label{sharpdecay}
Let $N=3$ and $2 <\sigma < \infty$ or $N=4$ and $1 < \sigma <\infty$. Assume that $u\in X$ is a nontrivial radial solution to 
\begin{equation}
\label{4NLSeqend}
 \gamma \Delta^2 u - \Delta u  = |u|^{2 \sigma}u ,\quad x\in\R^N .
\end{equation}
Then there exists $C\in \R\setminus\{0\}$ such that $$\lim_{|x|\to\infty}\frac{u(x)}{|x|^{2-N}}=C.$$ 
\end{prop}

The function space $X$, defined by \eqref{defX}, is the natural energy space associated to \eqref{alphazerointro}.
The proof of this proposition is based on a representation formula and a maximum principle for cooperative systems. 
Observe that as a direct corollary, there is no solution to \eqref{alphazerointro} with finite mass in dimensions $3$ and $4$. 
As it will appear later, this fact indirectly implies that the Lagrange multiplier $\alpha_c$ has to be strictly positive, see Corollary \ref{sign-la2}.

\medbreak

Next we enlighten a concentration behaviour of the ground state solutions to \eqref{Pc} when $\sigma N =4$ and $c \to c_N^*$. We remind that the existence of such ground states is guaranteed by Theorem \ref{thmmain}.
\begin{thm} \label{concentration}
Let $N \geq 1, \sigma N=4$, and $(c_n)_n \subset \R$ be a sequence satisfying, for any $n \in \N,$ $c_n > c^*_N$ with
$c_n \rightarrow c^*_N$ as $n \rightarrow \infty$, and $u_n$ be a ground state solution to \eqref{Pc}
for $c=c_n$ at level $\Gamma(c_n)$. Then there exist a sequence $(y_n)_n \subset \R^N$ and a least energy solution $u$
to the equation
\begin{align*}
\gamma \Delta^2 u + u =|u|^{\frac 8N}u,
\end{align*}
such that up to a subsequence,
$$
\left(\frac{\eps_n^4 c^*_N N}{4} \right)^{\frac N8}
u_n \left(\left(\frac{\eps_n^4 c^*_N N}{4}\right)^{\frac 14} x+ \eps_n y_n \right) \rightarrow u \ \text{in} \ L^q(\R^N)
\ \text{as} \ n \rightarrow \infty
$$
for $2 \leq q < 4^*$, where $\eps_n \to 0$ as $n \to \infty$.
\end{thm}
Theorem \ref{concentration} gives a description of ground state solutions to \eqref{Pc} as the mass $c_n$ approaches to $c_N^*$ from above. Roughly speaking, it shows that for $n \in \N$ large enough, we have
$$
u_n(x)\approx \left(\frac{4}{\eps_n^4 c^*_N N} \right)^{\frac N8}
u\left(\left(\frac{4}{\eps_n^4 c^*_N N} \right)^{\frac 14}\left(x- \eps_n y_n \right)\right).
$$
In the case where $\alpha \in \R$ is given in \eqref{4nls}, many additional properties are known on the least energy solutions.
In particular it is known that when  $\alpha >0$ is sufficiently small, all least energy solutions have a sign and are radial, see \cite[Theorem 3.9]{BoCaMoNa}. On the contrary when $\alpha \in \R$ is large, radial solutions are necessarily sign changing and, when $\sigma \in \N$, at least one least energy solution is radial, see \cite[Theorem 3.7, Corollary 3.8]{BoCaMoNa}. When looking to solutions with a prescribed mass, it is more delicate to deduce informations on the sign and symmetry of ground states. In that direction we only present the following result.
\begin{thm}\label{thm10}
Let $N \geq 2,$ $4 \leq \sigma N < 4^*$ and $\sigma \in \N$. There exists a $c_\pm >c_0$ such that, for any $c \in (c_0, c_\pm )$, \eqref{Pc} admits a ground state which is radial and sign changing.
\end{thm}
Positive radial solutions to \eqref{Pc} do exist as well. However we are not able to prove that those are ground states.
\begin{thm}\label{thm11}
Let $1\leq N\leq 4$ and  $4 \leq \sigma N < \infty $. For sufficiently large $c >0$, \eqref{Pc} admits a positive and radial solution.
\end{thm}

In the last part of the paper we investigate the dynamical  behaviour of the solutions to equation \eqref{4nlsdis}. The local well-posedness of the Cauchy problem is shown in \cite{Pa} for $ 0 <  \sigma N < 4^*$. In the mass subcritical case $0 < \sigma N < 4$,  the global existence for the Cauchy problem holds, see \cite{FiIlPa,Pa}  and it is conjectured that ground state solutions are orbitally stable. This is proved in \cite{BoCaMoNa} (see also \cite{NaPa}) under additional assumptions among which the fact that they are non degenerate, see \cite{BoCaMoNa} for a precise statement. We can show that, despite we are in the mass critical or mass supercritical cases, solutions to \eqref{4nlsdis} with initial data lying in some part of the space  
exist globally in time. We fix the notation 
$$
\mathcal{O}_c:=\{u\in S(c) : E(u)<\Gamma(c), \mbox{ and }  Q(u)>0  \}.
$$
\begin{thm}\label{globex}
Let $4 \leq \sigma N <4^*$. For any $c >c_0$, if the initial datum $\psi_0 \in 
\mathcal{O}_c
$, then the solution $\psi \in C([0,T); H^2 (\R^N))$ to \eqref{4nlsdis} 
exists globally in time.
\end{thm}
We also prove that the radial ground states 
are unstable by blow-up.

\begin{definition}\label{defunsta}
We say that a solution $ u \in H^2(\R^N)$ to \eqref{4nls} is unstable by blow-up in finite (respectively infinite) time if, for all $\varepsilon>0$, there exists $v\in H^2 (\R^N)$ such that 
$\|v - u \|_{H^2}<\varepsilon$ and the solution $\phi(t)$ to \eqref{4nlsdis} with initial data $\phi (0)=v$ blows up in finite (respectively infinite) time in the $H^2$ norm.
\end{definition}

\begin{thm} \label{unstable}
Let $4 \leq \sigma N < 4^*$ and $N \geq 2$. Then the standing waves associated to radial ground states 
are unstable by blow-up in finite or infinite time. Moreover, if $\sigma \leq 4$, then they are unstable by blow-up in finite time.
\end{thm}

In the case where $\alpha \in \R$ is fixed in \eqref{4nls} the fact that least energy solutions are unstable by blow-up in  finite time was recently established for $4 \leq \sigma N < 4^*$ in \cite{BoCaGoJe}. It should be noted that the results of \cite{BoCaGoJe} and of this paper are strongly based on arguments due to Boulenger and Lenzmann \cite{BoLe}. \medskip

We now describe the organization of the paper. In Section \ref{preresults} we present some preliminary results and give the proof of Theorem \ref{thm1}. In Section \ref{sectmanifold}, we establish some properties of the manifold $\mathcal{M}(c)$, and, in particular, we show that it is possible to find a Palais-Smale sequence $(u_n)_n\subset \mathcal{M}(c)$ for $E$ restricted to 
$\mathcal{S}(c)$, at level $\Gamma(c)$, see Lemma \ref{psbis}. In Section \ref{supercritical} we give the proof of 
Theorem \ref{thmmain}. Section \ref{behavior} is devoted to properties of the map $c \mapsto \Gamma(c)$ which are summarized in Theorem \ref{gammmaprop}.  In Section \ref{multiplicity}, we deal with radial solutions and establish Theorem \ref{thm} and Theorem \ref{radN34}. In Section  \ref{concentrationpart}  we prove the concentration result, namely Theorem \ref{concentration}. Section  \ref{special} contains the proofs of Theorem \ref{thm10} and Theorem \ref{thm11}. 
In Section  \ref{dispersiveeq}, we deal with the stability issues and prove Theorem \ref{globex} and Theorem \ref{unstable}. Finally in the Appendix we show that any solution $u \in H^2(\R^N)$ to (\ref{4nls}) satisfies the Derrick-Pohozaev identity $Q(u)=0$ and that all solutions of the limit problem (\ref{4NLSalpha=0}) belong to $H^2(\R^N)$ when $N \geq 5$.

\begin{notation}
The Sobolev space $H^2(\R^N)$ is endowed with its standard norm
$$
\|u\|^2 := \int_{\R^N}|\Delta u|^2 + |\nabla u|^2+|u|^2 \, dx.
$$
We use the notation $H^{-2}(\R^N)$ for the dual space to $H^2(\R^N)$. We denote by $'\rightarrow'$, respectively by $'\rightharpoonup'$, the strong convergence, respectively the  weak convergence in corresponding space, and denote by $B_R(x)$ the ball in $\R^N$ of center $x$ and radius $R>0.$  Throughout the paper we assume that $N \geq 1$ unless stated the contrary. We use the notation $o_n(1)$ for any quantity which tends to zero as $n \to \infty$.
\end{notation}

\section{Preliminary results and Proof of Theorem \ref{thm1}. } \label{preresults}

By interpolation and using the Sobolev inequalities, we infer that there exists $C_N(\sigma)>0$ such that for every $u \in H^2(\R^N)$,  (see e.g. \cite[Theorem in Lecture II]{Nir}),
\begin{equation}\label{G-N-H1-ineq}
 \|u\|^{2\sigma+2}_{2\sigma+2}\leq C_N(\sigma)\|\nabla u\|^{{\sigma N}}_2\|u\|^{2+\sigma(2-N)}_2,
\end{equation}
where
\begin{align*}
     \begin{cases}
     &0\leq\sigma, \hspace{.5cm} \mbox{if}\ N\leq 2,\\
     &0\leq\sigma<\dfrac{2}{N-2}, \hspace{.5cm} \mbox{if}\ N \geq 3,
     \end{cases}
\end{align*}
and (see for instance \cite[Lemma 2.1]{BoCaMoNa}),
\begin{equation}\label{G-N-H1-ineq2}
 \|u\|^{2\sigma+2}_{2\sigma+2}\leq C_N(\sigma)\|\nabla u\|^{N-(\sigma +1)(N-4)}_2\|\Delta u\|^{(N-2)(\sigma +1)-N}_2,
\end{equation}
where
\begin{align*}
     \begin{cases}
     &\dfrac{2}{N-2}\leq\sigma, \hspace{.5cm} \mbox{if}\ N = 3,4,\\
     &\dfrac{2}{N-2}\leq\sigma<\dfrac{4}{N-4}, \hspace{.5cm} \mbox{if}\ N \geq 5.
     \end{cases}
\end{align*}
We shall also often make use of the following interpolation inequality 
\begin{equation}\label{interpolation}
\int_{\R^N} |\nabla u|^2 dx \leq \left( \int_{\R^N}|\Delta u|^2 dx \right)^{\frac{1}{2}} \left( \int_{\R^N}| u|^2 dx \right)^{\frac{1}{2}}, \quad \mbox{for every} \quad u \in H^2(\R^N).
\end{equation}
Finally, for future reference, note that when $\sigma N =4$, one has
\begin{equation}\label{ajout2}
\frac{N}{N+4}\int_{\R^N}|u|^{2 +\frac{8}{N}} \, dx \leq  \left(\frac {c}{c^*_N}\right)^{\frac 4N}  \gamma \int_{\R^N}|\Delta u|^2 \, dx  \quad \mbox{for every} \quad u \in H^2(\R^N).
\end{equation}
Indeed, \eqref{ajout2} follows from the Gagliardo-Nirenberg inequality \eqref{G-N-H2-ineq} using the fact that 
$c_N^* = (\gamma C(N))^{\frac{N}{2}}$ where $C(N)$ is given in \eqref{cn},


\begin{lem} \label{sign-la}
Let $4 \leq \sigma N < 4^*$. Assume $u_c\in H^2(\R^N)$ solves
\begin{equation}\label{4nlsbis}
 \gamma \Delta^2 u - \Delta u + \alpha_c u = |u|^{2 \sigma}u
\end{equation}
with $\|u_c\|_2^2=c>0$. Then there exists $c_{N, \sigma}>0$ such that $ \alpha_c>0$ for any $c\in (0,c_{N, \sigma})$. Moreover,
\begin{enumerate}[(i)]
\item $c_{1,\sigma}=c_{2,\sigma}=\infty$ for every $\sigma\ge 4/N$, $ c_{3,\sigma}=\infty$ if  $4/3 \le\sigma \leq 2$ and $c_{4, 1}= \infty$;
\item  If $\sigma N=4$ then
$c_{N, \sigma} \geq  \left(\frac{N}{N-4}\right)^{\frac N4}c^*_N$ for $N \geq 5$.
\end{enumerate}
\end{lem}
\begin{proof}
We infer from Lemma \ref{Pohozaevs} in the Appendix that $Q(u_c)=0$. Therefore, we have 
\begin{equation}
\label{poho}
\gamma \int_{\R^N}|\Delta u_c|^2\, dx+ \frac 12\int_{\R^N}|\nabla u_c|^2\, dx=\dfrac{\sigma N}{2(2\sigma +2)}\int_{\R^N}|u_c|^{2\sigma+2}\, dx.
\end{equation}
Also multiplying \eqref{4nlsbis} by $u_c$ and integrating we get
\begin{align}\label{id1}
\gamma\int_{\R^N} |\Delta u_c|^2 \,dx +  \int_{\R^N}|\nabla u_c|^2 \, dx
 + \alpha_c \int_{\R^N}|u_c|^2\,dx=\int_{\R^N}|u_c|^{2\sigma +2} \, dx.
\end{align}
Combining \eqref{poho} and \eqref{id1} gives
\begin{equation}
\label{nonexe1}
c \alpha_c  = \gamma \left(\dfrac{4\sigma +4}{\sigma N}-1 \right)\int_{\R^N}|\Delta u_c|^2\, dx+ \left(\dfrac{2\sigma +2}{\sigma N} -1\right)\int_{\R^N}|\nabla u_c|^2\, dx.
\end{equation}
This obviously implies that $\alpha_c>0$ for any $c>0$ provided that either $N=1,2$ or $N=3$ with $\sigma \leq 2$ or $N=4$ with $\sigma =1$.

For the remaining cases, using the Gagliardo-Nirenberg inequality \eqref{G-N-H2-ineq}, we get from \eqref{poho} that
$$
\gamma \int_{\R^N}|\Delta u_c|^2\, dx \leq  \dfrac{\sigma N}{2(2\sigma +2)} B_N(\sigma) c^{1+\sigma(1 -\frac{N}{4})} \left(\int_{\R^N}|\Delta u_c|^2\, dx \right)^{\frac{\sigma N}{4}},
$$%
which implies
\begin{equation}
\label{nonexe2}
\left(\int_{\R^N}|\Delta u_c|^2\, dx \right)^{1 - \frac{\sigma N}{4}}\leq  \dfrac{\sigma N}{2(2\sigma +2)} \dfrac{B_N(\sigma)}{\gamma} c^{1+\sigma(1 -\frac{N}{4})}.
\end{equation}%
Thus, when  $4 < \sigma N < 4^*$, one deduces that
\begin{align}\label{limitc}
\int_{\R^N} |\Delta u_c|^2 \, dx \rightarrow \infty \,\, \text{as}\, \, c \rightarrow 0.
\end{align}
On the other hand,  using \eqref{interpolation}, we get from  \eqref{nonexe1}  that
\begin{equation}\label{divergence}
c\, \alpha_c  
\geq \gamma \left( \dfrac{4\sigma +4}{\sigma N}-1\right)\int_{\R^N}|\Delta u_c|^2\, dx +  \left(\dfrac{2\sigma +2}{\sigma N} -1\right) \left(\int_{\R^N}|\Delta u_c|^2\, dx \right)^{\frac{1}{2}} c^{\frac{1}{2}},
\end{equation}
and taking  \eqref{limitc} into account, it follows that $\alpha_c>0$ provided that $c>0$ is small enough. 

It remains to treat the case $\sigma N =4$ with $N \geq 5$.  Since (\ref{poho}) and (\ref{id1}) yield
\begin{align} \label{p3}
c\, \alpha_c 
= \gamma \int_{\R^N}|\Delta u_c|^2 \, dx - \frac{N-4}{N+4} \int_{\R^N}|u_c|^{2 + \frac 8N} \, dx,
\end{align}
  we deduce, from \eqref{ajout2},  that
$$
c\, \alpha_c  \geq \left(1-\frac {N -4}{N} \left(\frac {c}{c^*_N}\right)^{\frac 4N} \right) \gamma \int_{\R^N}|\Delta u_c|^2 \, dx.
$$
This shows $\alpha_c >0$ for $c <\left(\frac{N}{N-4}\right)^{\frac N4}c^*_N.$
\end{proof}

We now show that the two quadratic terms in the energy functional behave somehow in a similar manner. This observation is later used only to treat the case $\sigma N =4$ but we state it here under more general assumptions. 

\begin{lem}\label{boundcritical}
Let $4 \leq \sigma N < 4^*$ Assume that $(u_n)_n \subset S(c_n)$ is a sequence such that $(c_n)_n \subset \R$ is positive and bounded and $(E(u_n))_n \subset \R$ is bounded. Then
\begin{equation}\label{equivlapgrad0}
\left( \int_{\R^N} |\nabla u_n|^2 dx \right)_n \subset \R \quad  \mbox{is bounded  if and only if}  \quad \left( \int_{\R^N} |\Delta u_n|^2 dx \right)_n \subset \R \mbox{ is bounded}.  
\end{equation}
If, in addition, $(u_n)_n \subset \mathcal M(c_n)$, then,
\begin{enumerate}[(i)]
\item  if  $4 \leq \sigma N$ and $N=1,2$  or $4 \leq \sigma N < \frac{2N}{N-2}$ and $N = 3$, we have
\begin{equation}\label{equivlapgrad1}
 \frac1{c_n}\left( \int_{\R^N}|\nabla u_n|^2 \, dx \right)^2\le \int_{\R^N} |\Delta u_n|^2 \, dx  \le \frac{\sigma N C_N(\sigma)}{2\gamma( 2\sigma +2)}  c_n^{\frac{2-\sigma(N-2)}2}\left( \int_{\R^N}|\nabla u_n|^2 \, dx \right)^{\frac{\sigma N}2};
\end{equation}
\item if  $\frac{2N}{N-2}\leq \sigma N$ and $N=3,4$ or $ \sigma N < \frac{4N}{N-2}$  and $N \geq 5$, we have
\begin{equation}\label{equivlapgrad2}
 \frac1{c_n}\left( \int_{\R^N}|\nabla u_n|^2 \, dx \right)^2\le \int_{\R^N} |\Delta u_n|^2 \, dx  \le \left(\frac{\sigma N C_N(\sigma)}{2\gamma( 2\sigma +2)} \right)^{\frac2{4-\sigma(N-2)}} \left( \int_{\R^N}|\nabla u_n|^2 \, dx \right)^{\frac{4-\sigma(N-4)}{4-\sigma(N-2)}}.
\end{equation}
\end{enumerate}
 
\end{lem}
\begin{proof}
The left inequalities in \eqref{equivlapgrad1} and \eqref{equivlapgrad2} just come from the interpolation inequality (\ref{interpolation}). Together with the boundedness of the sequence $(c_n)_n$, they imply the reverse implication in \eqref{equivlapgrad0}.

We now focus on the direct implication in \eqref{equivlapgrad0}. Using the definition of $E$ and since $(E(u_n))_n$ is bounded, we infer that
\begin{equation}\label{boundkey}
\frac\gamma2\int_{\R^N} |\Delta u_n|^2 \, dx  +\frac12\int_{\R^N}|\nabla u_n|^2 \, dx \leq \frac{1}{ 2\sigma +2} \int_{\R^N} |u_n|^{2 \sigma + 2} \, dx + \sup_n | E(u_n)|.
\end{equation}
When  $N=1,2$ and $4 \leq \sigma N$ or $4 \leq \sigma N < \frac{2N}{N-2}$ and $N = 3$,
the claim immediately follows  from \eqref{G-N-H1-ineq}. If $N=3,4$ with $\frac{2N}{N-2}\leq \sigma N$ or $N \geq 5$ with $\frac{2N}{N-2} \leq \sigma N < \frac{4N}{N-2}$, we use instead 
(\ref{G-N-H1-ineq2}) to deduce the estimate 
\begin{multline}\label{eq:controledulaplacienparlegradient}
\gamma\int_{\R^N} |\Delta u_n|^2 \, dx +  \int_{\R^N}|\nabla u_n|^2 \, dx  \leq  \\\frac{C_N(\sigma)}{\sigma+1} \left( \int_{\R^N}|\nabla u_n|^2 \, dx \right)^{\frac{N}{2}- \frac{\sigma +1}{2}(N-4)} \left( \int_{\R^N}|\Delta u_n|^2 \, dx \right)^{\frac{N-2}{2}(\sigma +1) - \frac{N}{2}} +  \sup_n | E(u_n)| .
\end{multline}
The claim follows since $\frac{N-2}{2}(\sigma +1) - \frac{N}{2} <1$.
\medskip

If, in addition, $(u_n)_n \subset \mathcal M(c_n)$, then 
$$
\gamma\int_{\R^N} |\Delta u_n|^2 \, dx  +\frac12\int_{\R^N}|\nabla u_n|^2 \, dx = \frac{\sigma N}{2( 2\sigma +2)} \int_{\R^N} |u_n|^{2 \sigma + 2} \, dx.
$$
When  $N=1,2$ and $4 \leq \sigma N$ or $4 \leq \sigma N < \frac{2N}{N-2}$, we deduce from \eqref{G-N-H1-ineq} that 
$$
\gamma\int_{\R^N} |\Delta u_n|^2 \, dx  +\frac12\int_{\R^N}|\nabla u_n|^2 \, dx \le  \frac{\sigma N C_N(\sigma)}{2( 2\sigma +2)} c_n^{\frac{2-\sigma(N-2)}2}\left( \int_{\R^N}|\nabla u_n|^2 \, dx \right)^{\frac{\sigma N}2}
$$
so that the assertion (i) is proved. If  $\frac{2N}{N-2}\leq \sigma N$ and $N=3,4$ or $\frac{2N}{N-2} \leq \sigma N < \frac{4N}{N-2}$ and $N \geq 5$, we use \eqref{G-N-H1-ineq2} to deduce
$$
\gamma\int_{\R^N} |\Delta u_n|^2 \, dx  +\frac12\int_{\R^N}|\nabla u_n|^2 \, dx \le  \frac{\sigma N C_N(\sigma)}{2( 2\sigma +2)} \left( \int_{\R^N}|\nabla u_n|^2 \, dx \right)^{\frac{N-(\sigma+1)(N-4)}2}\left(\int_{\R^N} |\Delta u_n|^2 \, dx\right)^{\frac{(N-2)(\sigma+1)-N}2},
$$
which leads to assertion (ii).
\end{proof}

\begin{remark}\label{Rem:equivlapgrad}
 We emphasize that when $\sigma N =4$, and whatever $N \geq 1$, we deduce from the preceding lemma that for every $c>0$, there exists $C>0$ such that 
\begin{equation}\label{equivlapgradsigmaN=4}
 \frac1{c}\left( \int_{\R^N}|\nabla u|^2 \, dx \right)^2\le \int_{\R^N} |\Delta u|^2 \, dx  \le C\left( \int_{\R^N}|\nabla u|^2 \, dx \right)^{2},
\end{equation}
 for every $u\in \mathcal M(c)$.
\end{remark}

We end this section by proving the non existence result stated in Theorem \ref{thm1}.

\begin{proof}[Proof of Theorem \ref{thm1}]
Observe first that $m(c) \leq 0$ for any $c>0$. Indeed, it follows from \eqref{dilatation}-\eqref{functional} that, for any $u \in S(c)$, $E(u_\lambda) \rightarrow 0$ as $ \lambda \rightarrow 0^+$. 
Now, using \eqref{ajout2}, we have for any $u \in S(c)$,
\begin{align} \label{none}
\begin{split}
E(u)& \geq \frac {\gamma}{2}\int_{\R^N}|\Delta u|^2\, dx 
     - \frac{N}{2N+8}\int_{\R^N}|u|^{2 + \frac 8N} \, dx \\
    & \geq \frac {\gamma}{2} \left(1-\left(\frac{c}{c_N^*}\right)^{\frac 4N}\right) \int_{\R^N}|\Delta u|^2 \, dx.
\end{split}
\end{align}
Therefore, we deduce that 
$m(c) \geq 0$, whence $m (c)=0$, for $c \leq c^*_N:=(\gamma C(N))^{\frac N4}$.

Next we prove that 
there is no solution to \eqref{Pc} when $c \leq c^*_N$. Indeed if $u$ is solution to  \eqref{Pc}, then $Q(u)=0$ and applying \eqref{ajout2}, we get
$$
\gamma \int_{\R^N}|\Delta u|^2\, dx + \frac 12\int_{\R^N} |\nabla u|^2 \, dx = \frac {N}{N+4}\int_{\R^N}|u|^{2 + \frac 8N} \, dx \leq  \left(\frac{c}{c^*_N}\right)^{\frac 4N} \gamma \int_{\R^N}|\Delta u|^2 \, dx,
$$
which implies that $u=0$ because $c\leq c^*_N.$

Finally let us prove that $m(c)=-\infty$ for $c>c^*_N$. It follows from \cite{BoLe}, see also \cite{BeFrVi}, that the best constant $B_N(\frac 4N)$ in \eqref{G-N-H2-ineq} is achieved, i.e. there exists a $U \in H^2(\R^N)$ satisfying
\begin{align} \label{u}
\|U\|_{2 + \frac 8N}^{2 + \frac 8N}=B_N(\frac 4N)\|U\|_2^{\frac 8N}\|\Delta U\|_2^{2}.
\end{align}
Choosing
\begin{align} \label{defw}
w:=c^{\frac 12} \frac{U}{\|U\|_2} \in S(c),
\end{align}
and taking the identity \eqref{u} into account, we get
\begin{align} \label{key10}
E(w_{\lambda}) &=\frac {c}{2\|U\|_2^2}\lambda^2 \gamma\int_{\R^N}|\Delta U|^2\, dx
     +\frac {c}{2\|U\|_2^2}\lambda \int_{\R^N} |\nabla U|^2 \, dx\nonumber \\
     &-\frac{N}{2N+8} \left(\frac {c^{\frac 12}}{\|U\|_2}\right)^{2 + \frac 8N} \lambda^2\int_{\R^N}|U|^{ 2 + \frac 8N} \, dx\\
     &= \frac {c}{2\|U\|_2^2} \gamma \left(1-\left(\frac{c}{c_N^*}\right)^{\frac 4N}\right)\lambda^2
     \int_{\R^N}|\Delta U|^2\, dx
     + \frac {c}{2\|U\|_2^2}\lambda\int_{\R^N} |\nabla U|^2 \, dx.\nonumber
\end{align}
This clearly implies $ E(w_{\lambda}) \rightarrow -\infty$ as $\lambda \rightarrow \infty$ for $c>c^*_N$.  
\end{proof}

\section{Some Properties of the constraint $\mathcal{M}(c)$} \label{sectmanifold}


In this section, we work out some important properties of the manifold $\mathcal{M}(c)$ and of the energy functional $E$ constrained to $\mathcal{M}(c)$. 
Since  $Q(u)=0$ for any $u \in \mathcal{M}(c)$, we can write
\begin{equation}
\label{relEQ1}
E(u)=E(u)-\frac{2}{\sigma N}Q(u)=\gamma\dfrac{\sigma N-4}{2\sigma N}\int_{\R^N}|\Delta u|^2\, dx +\dfrac{\sigma N -2}{2\sigma N}\int_{\R^N}|\nabla u|^2\, dx.
\end{equation}
We shall repeatedly use this relation in the sequel. 

\begin{lem}\label{Mnonvoid}
Assume that $ 4 \leq \sigma N < 4^*$. If $c >c_0$ then  $\mathcal{M}(c) \neq \emptyset$.
\end{lem}
\begin{proof}
If $4 < \sigma N < 4^*$, the property that $\mathcal{M}(c) \neq \emptyset$ for any $c >0$ follows from the observation that, in \eqref{functional}, $E(u_{\lambda})$ is increasing for $\lambda >0$ small and goes to $- \infty$ as $\lambda \to + \infty$. This implies indeed  that the function $\lambda \mapsto E(u_{\lambda})$ has a least a local maximum, corresponding thus to an element of  $\mathcal{M}(c)$, see \eqref{derivationofQ}. 
If $4 = \sigma N$, we also observe in \eqref{key10} that, for $c > c_N^*$ fixed, the function $\lambda \mapsto E(w_{\lambda})$ is increasing for $\lambda >0$ small and goes to $- \infty$ as $\lambda \to + \infty$. We then conclude as in the previous case.
\end{proof}

We  say that $E$ restricted to $\mathcal{M}(c)$ is coercive if for any $a \in \R$ the subset  $\{ u \in \mathcal{M}(c) : E(u) \leq a \}$ is bounded.

\begin{lem}\label{coercive}
Let $4 \leq \sigma N < 4^*$ and $c >c_0$, then $E$ restricted to $\mathcal{M}(c)$ is coercive and bounded from below by a positive constant. 
\end{lem}
\begin{proof}
We assume throughout the proof that $u\in \mathcal{M} (c)$. 
In view of the expression of $E(u)$ given by \eqref{relEQ1}, the coercivity trivially holds when $\sigma N >4$, whereas the conclusion follows from Lemma \ref{boundcritical} when $\sigma N =4$. 

We now show the existence of a positive lower bound. 
When $\sigma N>4$,  using \eqref{poho} and the Gagliardo-Nirenberg inequality \eqref{G-N-H2-ineq} we get that
\begin{align}\label{lowerboundlap}
\gamma \int_{\R^N}|\Delta u|^2 \, dx &\leq \gamma \int_{\R^N}|\Delta u|^2 \, dx +\dfrac{1}{2} \int_{\R^N}|\nabla u|^2 \, dx
                                = \frac {\sigma N}{2(2\sigma +2)}  \int_{\R^N}|u|^{2\sigma +2} \, dx\nonumber \\
                                & \leq \frac{ \sigma N  B_N(\sigma)}{2(2\sigma+2)} c^{1+\sigma - \sigma N/4} \left(\int_{\R^N}|\Delta u|^2 \, dx \right)^{\frac{\sigma N}{4}}.
\end{align}
Since $\sigma N>4$, this provides a lower bound on $\int_{\R^N}|\Delta u|^2 \, dx$, whence on $E$. 

When $\sigma N = 4$, we need a lower bound on $\int_{\R^N}|\nabla u|^2 \, dx$. 
Note that, still by \eqref{poho},
\begin{equation}\label{ajout1}
\int_{\R^N}|\nabla u|^2 \, dx \leq \frac{N}{N+4}\int_{\R^N}|u|^{2 + \frac 8N} \, dx 
\end{equation}
If $1 \leq N \leq 3$, we have $2 + \frac 8N < 2^*$, so that,  the Gagliardo-Nirenberg inequality \eqref{G-N-H1-ineq} provides, for some constant $C >0$,
the estimate
$$
\int_{\R^N}|u|^{2 + \frac 8N} \, dx  \leq C \left(\int_{\R^N}|\nabla u|^2 \, dx\right)^2,
$$
which, in view of  \eqref{ajout1}
allows
to conclude. When $N \geq 4$, we have $2^*< 2 + \frac 8N < 4^*$ so that the Gagliardo-Nirenberg inequality
\eqref{G-N-H1-ineq2} yields, for some constant $C >0$,
\begin{equation*}
\int_{\R^N}|u|^{2 +\frac 8N } \, dx\leq C_N(\sigma)  \left(\int_{\R^N}|\nabla u|^2 \, dx\right)^{\frac 8N}\left(\int_{\R^N}|\Delta u|^{2}\,dx\right)^{1-\frac4N}\le C  \left(\int_{\R^N}|\nabla u|^2 \, dx\right)^{2},
 \end{equation*}
where the last inequality follows from Remark \ref{Rem:equivlapgrad}. In view of \eqref{ajout1} the conclusion again follows. 
\end{proof}

\begin{remark}\label{lowerboundsinMc}
 It is clear from Lemma \ref{boundcritical} and Lemma \ref{coercive} that there exists a $\delta >0$ such that
 $$ \inf_{u\in \mathcal M(c)}\min\left(E(u),\int_{\R^N}|u|^{2\sigma +2} \, dx,\int_{\R^N}|\Delta u|^2 \, dx\right) \geq \delta >0$$
if $4 \leq \sigma N < 4^*$ and $c >c_0$ and
$$  \inf_{u\in \mathcal M(c)}\int_{\R^N}|\nabla u|^2\geq \delta >0$$
if $ 4 \leq  \sigma N $ if $N=1,2,3,4$ and $4 \leq \sigma N < \frac{4N}{N-2}$ if $N \geq 5.$
\end{remark}

\begin{lem}\label{unique} Let $4 \leq \sigma N < 4^*$, $c>0$ and $u \in S(c)$.
 If $\sigma N=4$, assume further that $\sup_{\lambda>0} E(u_{\lambda}) <  \infty$. 
There exists a unique $\lambda_{u} >0$ for which $ u_{\lambda_{u}} \in \mathcal{M}(c)$. Moreover, we have $ E(u_{\lambda_{u}})=\max_{\lambda >0}E(u_{\lambda})$ and  the function $\lambda \mapsto E(u_{\lambda})$ is concave on $[\lambda_{u}, \infty)$. Moreover, if $Q(u)\leq 0$, then $\lambda_{u} \in (0,1]$.
\end{lem}

\begin{proof}
For any $u \in S(c)$, differentiating the identity \eqref{functional} with respect to $\lambda>0$, we obtain
\begin{align*}
\dfrac{d}{d\lambda}E(u_{\lambda})&=\gamma\lambda\int_{\R^N}|\Delta u|^2\, dx+\frac{1}{2}\int_{\R^N}|\nabla u|^2\, dx-\frac{\sigma N\lambda^{\sigma N/2 -1}}{2(2\sigma+2)}\int_{\R^N}|u|^{2\sigma+2}\, dx\\
&= \dfrac{1}{\lambda}Q(u_\lambda).
\end{align*}
When $\sigma N >4$ it is easily seen that there exists a unique $\lambda_{u} >0$ such that $Q(u_{\lambda_u})=0$
and also that 
\begin{equation}\label{maxr}
\dfrac{d}{d\lambda}E (u_\lambda ) >0 \, \mbox{ if } \, \lambda \in (0,\lambda_{u}) \quad \mbox{and} \quad \dfrac{d}{d\lambda}E (u_\lambda ) < 0  \, \mbox{ if } \, \lambda \in (\lambda_{u}, \infty)
\end{equation}
from which we deduce that $E(u_\lambda)< E(u_{\lambda_{u}})$, for any $\lambda>0$, $\lambda \neq \lambda_{u}$. When $\sigma N =4$, since we assume that $\sup_{\lambda>0} E(u_{\lambda}) <  \infty$, then necessarily
\begin{align} \label{s}
\gamma \int_{\R^N}|\Delta u|^2\, dx <\frac {1}{\sigma +1} \int_{\R^N} |u|^{2 \sigma +2 } \, dx,
\end{align}
and thus there also exists a unique $\lambda_{u} >0$ such that $Q(u_{\lambda_{u}})=0$ and (\ref{maxr}) holds.

Now writing 
$\lambda = t \lambda_{u}$, we have
\begin{align*}
\dfrac{d^2}{d^2\lambda}E(u_{\lambda})&=\gamma \int_{\R^N}|\Delta u|^2\, dx - \frac{\sigma N (\sigma N -2)}{4 (2 \sigma +2)} t^{\frac{\sigma N}{2}-2} \lambda_{u}^{\frac{\sigma N}{2}-2} \int_{\R^N}|u|^{2\sigma+2}\, dx\\
&= \dfrac{1}{\lambda_{u}^2}\Big[ \gamma \lambda_{u}^2 \int_{\R^N}|\Delta u|^2\, dx - \frac{\sigma N (\sigma N -2)}{4 (2 \sigma +2)} t^{\frac{\sigma N}{2}-2} \lambda_{u}^{\frac{\sigma N}{2}} \int_{\R^N}|u|^{2\sigma+2}\, dx\Big].
\end{align*}
Since
$$0 = Q(u_{\lambda_{u}}) = \gamma \lambda_{u}^2 \int_{\R^N}|\Delta u|^2\, dx + \frac{1}{2} \lambda_{u} \int_{\R^N}|\nabla u|^2\, dx - \frac{\sigma N }{2 (2 \sigma +2)} \lambda_{u}^{\frac{\sigma N}{2}} \int_{\R^N}|u|^{2\sigma+2}\, dx,$$
we infer that $ \displaystyle \frac{d^2}{d^2 \lambda} E(u_{\lambda}) <0$ for any $t \geq 1$. This proves the lemma.  
\end{proof}

\begin{remark}\label{regularitelambda}
It is easily seen that the map $u\to \lambda_u$ is of class $C^1$. This follows by the Implicit Function Theorem if $\sigma N >4$, while  the case $\sigma N =4$ is even simpler since we then have the explicit expression of $\lambda_u$.
\end{remark}

Our next lemma is not needed to derive our main results but it provides a better understanding of the set $\mathcal{M}(c)$ and could prove useful in other contexts.

\begin{lem}\label{manifold}
Assume $ 4 \leq \sigma N< 4^*$ and $c>0$. Then  $\mathcal{M}(c)$ is a $C^1$ manifold of codimension $2$ in $H^2 (\R^N)$, whence a $C^1$ manifold of codimension $1$ in $S(c)$.
\end{lem}
\begin{proof}
By definition, $u \in \mathcal{M}(c)$ if and only if $G(u):=\|u\|_2^2 -c =0$ and $Q(u)=0.$ It is easy to check that $G, Q$ are of $C^1$ class. Hence we only have to check that for any $u \in \mathcal{M}(c)$,
$$
(dG(u), dQ(u)): H^2 (\R^N ) \rightarrow \R^2 \,\, \text{is surjective}.
$$
Otherwise, $dG(u)$ and $dQ(u)$ are linearly dependent, i.e. there exists $\nu \in \R$ such that for any $\varphi \in H^2 (\R^N)$,
$$
2\gamma \int_{\R^N} \Delta u \Delta \varphi \, dx +  \int_{\R^N} \nabla u \cdot \nabla \varphi \, dx
-\frac{\sigma N}{2} \int_{\R^N}|u|^{2\sigma}u \varphi \, dx= 2 \nu \int_{\R^N}u \varphi\,dx,
$$
namely $u$ weakly solves
$$
2\gamma \Delta^2 u - \Delta u =2\nu u +\dfrac{\sigma N}{2} |u|^{2\sigma}u.
$$
From Lemma \ref{Pohozaevs} in the Appendix, we deduce that
$$
4\gamma \int_{\R^N}|\Delta u|^2 \, dx +   \int_{\R^N}|\nabla u|^2 \, dx
= \frac {\left(\sigma N\right)^2}{2(2\sigma +2)} \int_{\R^N} |u|^{2\sigma +2} \, dx,
$$
and then, since  $Q(u) =0$, we infer that
$$
4 \gamma \int_{\R^N}|\Delta u|^2 \, dx +  \int_{\R^N}|\nabla u|^2 \, dx
= \sigma N  \gamma\int_{\R^N}|\Delta u|^2 \, dx
+  \frac{\sigma N}{2} \int_{\R^N}|\nabla u|^2 \, dx
$$
which is impossible since  $\sigma N \geq 4$ and $u \in S(c)$.
\end{proof}

Our aim now is to prove the existence of a Palais-Smale sequence $(u_n)_n\subset \mathcal M(c)$  at level $\Gamma(c)$ for $E$ restricted $S(c)$. Our arguments are directly inspired from \cite{BaSo2}. We start be recalling the following definition  \cite[Definition 3.1]{Ghoussoub}.

\begin{definition}
Let $B$ be a closed subset of a metric space $Y$. We say that a class $\mathcal{G}$ of compact subsets of $Y$ is a homotopy stable family with closed boundary $B$ provided
\begin{enumerate}
\item every set in $\mathcal{G}$ contains $B$;
\item for any $A\in \mathcal{G}$ and any $\eta \in C([0,1]\times Y, Y)$ satisfying $\eta (t,x)=x$ for all $(t,x)\in (\{0\}\times Y)\cup ([0,1]\times B)$, we have $\eta (\{1\} \times A)\in \mathcal{G}$.
\end{enumerate}
\end{definition}
We explicitly observe that $B=\varnothing$ is admissible. 
Now for $4 < \sigma N < 4^*$, we define $F: S(c) \mapsto \R$ by $F(u)=E(u_{\lambda_u})$, where $\lambda_u$ is uniquely defined by $u_{\lambda_u}\in \mathcal{M}(c)$, see Lemma \ref{unique}. When $\sigma N=4$, we define similarly $F$ on the open set $\mathcal{E}(c) \subset S(c)$ on which $\sup_{\lambda >0}E(u_{\lambda}) < \infty.$  As in \cite[Lemma 3.7]{BaSo2} it can be readily proved that if $(u_n)_n \subset \mathcal{E}(c)$ is such that $u_n \to u \in \partial \mathcal{E}(c)$ as $n \to  \infty$ then $F(u_n) \to + \infty$ as $n \to \infty$. Finally note that since the map $u \mapsto \lambda_u$ is of class $C^1$, see Remark \ref{regularitelambda}, the functional $F$ is of class $C^1$.

\begin{lem}
\label{ps}
Let $4 \leq \sigma N < 4^*$.
Let $\mathcal{G}$ be a homotopy stable family of compact subsets of $S(c)$ with closed boundary $B$ and let
$$e_{\mathcal{G}}:= \inf_{A\in \mathcal{G}}\max_{u\in A}F(u).$$ 
Suppose that $B$ is contained in a connected component of $\mathcal{M}(c)$ and that $\max\{\sup F(B),0\}<e_{\mathcal{G}}<\infty$. Then there exists a Palais-Smale sequence $(u_n)_n \subset \mathcal{M}(c)$ for $E$ restricted to $S(c)$ at level $e_{\mathcal{G}}$.
\end{lem}
 
\begin{proof}
Take $(D_n)_n \subset \mathcal{G}$ such that $\max_{u\in D_n}F(u)<e_{\mathcal{G}}+\frac1n$ and
$$\eta :[0,1]\times S(c)\rightarrow S(c),\ \eta (t,u)=u_{{1-t + t \lambda_u}}.$$
Since $\lambda_u =1$ for any $u\in \mathcal{M}(c)$, and $B\subset \mathcal{M}(c)$, we have $\eta (t,u)=u$ for $(t,u)\in (\{0 \}\times S(c))\cup ([0,1]\times B)$. Observe also that $\eta$ is continuous. Then, using the definition of $\mathcal{G}$, we have
$$A_n:= \eta (\{1 \}\times D_n )=\{u_{\lambda_u}:\ u\in D_n \}\in \mathcal{G}.$$
Also notice that $A_n \subset \mathcal{M}(c)$ for all $n$. Let $v\in A_n$, i.e. $v=u_{\lambda_u}$ for some $u\in D_n$ and $F(u)=F(v)$. So $\max_{A_n}F=\max_{D_n}F$ and therefore $(A_n)_n \subset \mathcal{M}(c)$ is another minimizing sequence of $e_{\mathcal{G}}$. Using the equivariant minimax principle \cite[Theorem 3.2]{Ghoussoub}, we obtain a Palais-Smale sequence $(\tilde{u}_n)_n$ for $F$ on $S(c)$ at level $e_{\mathcal{G}}$ such that $\text{dist}_{H^2 (\R^N )} (\tilde{u}_n , A_n)\rightarrow 0$ as $n\rightarrow \infty$. Now  writing $\lambda_n=\lambda_{\tilde{u}_n}$ to shorten the notations, we set $u_n=(\tilde{u}_n)_{\lambda_n}\in \mathcal{M}(c)$.  We claim that there exists $C>0$ such that,
 \begin{equation}
\label{BaSoe1}
\frac1C \leq \lambda_n^2\leq C
\end{equation}
 for $n \in \N$ large enough. 
Indeed, notice first that
 \begin{equation}
\label{BaSoe2}
\lambda_n^2=\dfrac{\int_{\R^N} |\Delta u_n|^2  \, dx}{\int_{\R^N} |\Delta \tilde{u}_n |^2  \, dx}.
\end{equation}
Since by definition we have $E(u_n)=F(\tilde{u}_n )\rightarrow e_{\mathcal{G}}$, we deduce from Lemma \ref{coercive} and Remark \ref{lowerboundsinMc}, that there exists $M>0$ such that 
 \begin{equation}
\label{BaSoe3}
\frac1M\leq \|\Delta u_n\|_{2}\leq M.
\end{equation}
On the other hand, since $A_n \subset \mathcal{M}(c)$, is a minimizing sequence for $e_{\mathcal{G}}$ and $E$ is coercive on $\mathcal{M}(c)$, we deduce that $(A_n)_n$ is uniformly bounded in $H^2 (\R^N)$ and thus from $\text{dist}_{H^2 (\R^N )} (\tilde{u}_n , A_n)\rightarrow 0$ as $n\rightarrow \infty$, this implies that $\sup_n\|\tilde{u}_n \| <\infty$. Also, since $A_n$ is compact for every $n \in \N$, there exists a $v_n \in A_n$ such that  $\text{dist}_{H^2 (\R^N )} (\tilde{u}_n , A_n)=\|v_n-\tilde{u}_n \|$  and, using once again Remark \ref{lowerboundsinMc}, we also infer that
$$\|\Delta \tilde{u}_n \|_{2}\geq \|\Delta v_n \|_{2}- \|\Delta (\tilde{u}_n - v_n) \|_{2}\geq 
\dfrac{\delta}{2},$$
for some $\delta >0$.
This proves the claim.

Next, we show that $(u_n)_n \subset \mathcal{M}(c)$ is a Palais-Smale sequence for $E$ on $S(c)$ at level $e_{\mathcal{G}}$.
Denoting by $\|.\|_\ast$ the dual norm of $(T_{u_n}S(c))^\ast$, we have
\begin{equation}\label{start}
\|dE(u_n)\|_\ast=\sup_{\psi \in T_{u_n}S(c),\ \|\psi\| \leq 1}|dE(u_n)[\psi]|=\sup_{\psi \in T_{u_n}S(c),\ \|\psi\| \leq 1}|dE(u_n)[ (\psi_{\frac1{\lambda_n}})_{\lambda_n} ]|.
\end{equation}
It can be checked that the map $T_u S(c)\rightarrow T_{u_{\lambda_u}} S(c)$ defined by $\varphi \rightarrow \varphi_{\lambda_u}$ is an isomorphism, see \cite[Lemma 3.6]{BaSo2} for a proof of a closely related result. Also, see here \cite[Lemma 3.8]{BaSo2} or \cite[Lemma 3.2]{BaSo1}, we have that $dF(u)[\varphi]=dE(u_{\lambda_u}) [\varphi_{\lambda_u}]$ for any $u\in S(c)$ and $\varphi \in T_u S(c)$. It follows that
\begin{equation}\label{lienclef}
\|dE(u_n)\|_\ast = \sup_{\psi \in T_{u_n}S(c),\ \|\psi\| \leq 1} |dF(\tilde{u}_n)[\psi_{\frac1{\lambda_n}}]|.
\end{equation}
At this point it is easily seen from \eqref{BaSoe1} that (increasing $C$ if necessary)
$ \|\psi_{\frac1{\lambda_n}}\| \leq C \|\psi\| \leq C$
and we deduce from \eqref{lienclef} that $(u_n)_n \subset \mathcal{M}(c)$ is a Palais-Smale sequence for $E$ on $S(c)$ at level $e_{\mathcal{G}}$.
\end{proof}

\begin{lem}\label{psbis}
Let $ 4\leq \sigma N< 4^*$ and  $c >c_0 $. There exists a Palais-Smale sequence  $(u_n)_n \subset \mathcal{M}(c)$ for $E$ restricted to $S(c)$ at the level $\Gamma(c)$. 
\end{lem}

\begin{proof}
We use Lemma \ref{ps} taking the set  $\bar{\mathcal{G}}$ of all singletons belonging to $S(c)$ (or to $\mathcal{E}(c) \subset S(c)$ if $\sigma N =4$) and $B=\varnothing$. It is clearly a homotopy stable family of compact subsets of $S(c)$ (without boundary).  Observe that 
$$e_{\bar{\mathcal{G}}}:=\inf_{A\in \bar{\mathcal{G}}}\max_{u\in A}F(u)=\inf_{u\in S(c)}\sup_{\lambda >0} E(u_\lambda).$$
We claim that
\begin{align} \label{minmax}
e_{\bar{\mathcal{G}}}=\Gamma(c).
\end{align}
Indeed, on one hand, we observe that for any $u \in S(c)$, either  $\sup_{\lambda>0} E(u_{\lambda}) = +  \infty$ or there exists $\lambda_{u} >0$ such that $ u_{\lambda_{u}} \in \mathcal{M}(c)$
and $ E(u_{\lambda_{u}}) \leq \sup_{\lambda >0}E(u_{\lambda})$. This implies that
$$
\inf_{u\in S(c)}\sup_{\lambda>0} E(u_\lambda) \geq \inf_{u \in \mathcal{M}(c)}E(u).
$$
On the other hand, for any $u \in \mathcal{M}(c)$, $ E(u)  \geq \sup_{\lambda >0}E(u_{\lambda})$ and so 
$$ \inf_{u \in \mathcal{M}(c)}E(u) \geq \inf_{u\in S(c)}\sup_{\lambda>0} E(u_\lambda).$$
Thus (\ref{minmax}) holds and the lemma follows directly from Lemma \ref{ps}.
\end{proof}

\begin{lem}
\label{conv}
Let $4 \leq \sigma N < 4^*$, $c>c_0$ and $(u_n)_n \subset \mathcal{M}(c)$ be a Palais-Smale sequence for $E$ restricted to $S(c)$. Then there exist $u_c \in H^2(\R^N)$, a sequence $(\alpha_n)_n \subset \R$ and $\alpha_c \in \R$  such that, up to translation and up to the extraction of a subsequence,
\begin{enumerate}[(i)]
\item $u_n \rightharpoonup u_c \neq 0$ in $H^2(\R^N)$ as $n \to \infty$;
\item $\alpha_n \rightarrow \alpha_c$ in $\R$ as $n \to \infty$;
\item $\gamma \Delta^2 u_n- \Delta u_n + \alpha_n u_n- |u_n|^{2 \sigma}u_n \rightarrow 0$ in $H^{-2}(\R^N)$ as $n \to \infty$;
\item $\gamma \Delta^2 u_c -\Delta u_c +\alpha_c u_c=|u_c|^{2\sigma}u_c$.
\end{enumerate}
Here $H^{-2}(\R^N)$ denotes the dual of $H^2(\R^N)$.
In addition, if $\|u_n - u_c\|_{2 \sigma +2} \to 0$  and $\alpha_c >0$, then $\|u_n - u_c\| \to 0$ as $n \to \infty$. 
\end{lem}

\begin{proof}
First observe that, because of Lemma \ref{coercive}, we can assume without loss of generality that $(u_n)_n \subset \mathcal{M}(c)$ is a bounded sequence. After a suitable translation in $\R^N$ and up to the extraction of a subsequence, we can also assume that $u_n \rightharpoonup u_c \neq 0$. Indeed, if $u_c=0$, then, applying \cite[Lemma I.1]{Li2}, we infer that
$$\lim_{n\to \infty}\int_{\R^N}|u_n|^{2\sigma+2}\, dx =0,$$
which contradicts Remark \ref{lowerboundsinMc}.
 
Now since,  $(u_n)_n$ is bounded in $H^2(\R^N)$, we know from \cite[Lemma 3]{BeLi2} (adapted from the unit sphere to $S(c)$), that
$\| dE_{|_{S(c)}} (u_n)\|_{H^{-2}}=o_n(1)$ is equivalent to $\| dE (u_n) - \frac{1}{c}dE(u_n)[u_n] u_n\|_{H^{-2}}=o_n(1) $. Therefore, for any $\varphi \in H^2(\R^N)$, we infer that
\begin{align} \label{converge}
\gamma\int_{\R^N}\Delta u_n \Delta \varphi\, dx +\int_{\R^N}\nabla u_n \nabla \varphi \, dx  + \alpha_n
\int_{\R^N}u_n \varphi\, dx - \int_{\R^N}|u_n|^{2\sigma}u_n \varphi\, dx =o_n(1),
\end{align}
where
\begin{align} \label{alpha}
-\alpha_n =\frac{1}{c}\left(\gamma\int_{\R^N}|\Delta u_n|^2\, dx+\int_{\R^N}|\nabla u_n|^2\, dx-\int_{\R^N}|u_n|^{2\sigma+2}\, dx\right).
\end{align}
From \eqref{converge}-\eqref{alpha}, we deduce that $(ii)$-$(iii)$ hold whereas the weak convergence $u_n \rightharpoonup u_c$ in $H^2(\R^N)$ implies,  in a standard way, $(ii)$-$(iii)$ and $(iv)$.

Finally, assume further that $(u_n)_n $ strongly converges to $u_c$ in $L^{2 \sigma +2}(\R^N)$.
Recalling that $(u_n)_n $ is bounded in $H^2(\R^N)$ and using the strong convergence in 
$L^{2 \sigma +2}(\R^N)$, 
it follows from $(ii)$-$(iii)$ and $(iv)$ that
\begin{align}\label{conv1}
\begin{split}
&\gamma \int_{\R^N} |\Delta u_n|^2 \, dx + \int_{\R^N}|\nabla u_n|^2 \, dx + \alpha_c\int_{\R^N} |u_n|^2 \, dx \\
&= \gamma \int_{\R^N} |\Delta u_c|^2  \, dx + \int_{\R^N} |\nabla u_c|^2 \, dx + \alpha_c \int_{\R^N} |u_c|^2 \, dx + o_n(1).
\end{split}
\end{align}
Since $\alpha_c>0$ and we already know that  $u_n \rightharpoonup u_c$ in $H^2(\R^N)$, this implies 

that $u_n \rightarrow u$ in $H^2(\R^N)$ as $n \to \infty$. Thus the proof is complete.
\end{proof}

\section{Existence of ground states, proof of Theorem \ref{thmmain}} \label{supercritical}

In this section we give the proof of Theorem \ref{thmmain}. We start by a lemma which completes Lemma \ref{conv}. 
\begin{lem} \label{propcompactness}
Let $4 \leq \sigma N < 4^*$, $c >c_0$ and $(u_n)_n \subset \mathcal{M}(c)$ be a Palais-Smale sequence for $E$ restricted to $S(c)$, at the level $\Gamma(c)$. Assume that  $u_n\rightharpoonup u_c \neq 0$ in $H^2(\R^N)$. If
\begin{equation}\label{lemgammadecrbbb}
\Gamma(c) \leq \Gamma(\bar c) \ for \ any \ \bar c \in (c_0,c],
\end{equation}
then $\|u_n -u_c\|_{2 \sigma +2}\rightarrow 0$ as $n \to \infty$ and $E(u_c) = \Gamma(c)$. 
\end{lem}
\begin{proof}
We first prove the strong convergence in $L^{2 \sigma +2}(\R^N)$ and the equality $E(u_c) = \Gamma(c)$. 
By Lemma \ref{conv} we know that there exists a  $\alpha_c \in \R$ such that $u_c$ satisfies \eqref{4nls} and thus $Q(u_c)=0$ by Lemma \ref{Pohozaevs}. By weak lower semicontinuity, we infer that $0 <c_1:=\|u_c\|_2^2 \leq c$. 
. 
Since $u_n \rightharpoonup u_c$ in  $H^2(\R^N)$ as $n \to \infty$, we have 
\begin{align*} \label{bl1}
\begin{split}
&\|\Delta (u_n -u_c) \|_{2}^2 +\|\Delta u_c \|_{2}^2 =\|\Delta  u_n  \|_{2}^2 +o_n(1), \\
&\|\nabla (u_n -u_c) \|_{2}^2 +\|\nabla u_c \|_{2}^2 =\|\nabla  u_n  \|_{2}^2 +o_n(1),\\
\end{split}
\end{align*}
while Brezis-Lieb's Lemma implies 
\begin{equation*}\label{bl2}
\| u_n -u_c \|_{2 \sigma +2}^{2\sigma+2} +\| u_c \|_{2 \sigma +2}^{2\sigma+2} =\|u_n  \|_{2 \sigma +2}^{2\sigma+2} +o_n(1).
\end{equation*}
Since $Q(u_c)=0,$ and $Q(u_n)=0$, it then follows 
that $Q(u_n-u_c)=o_n(1)$, as well as
\begin{align} \label{E}
E(u_n- u_c)+E(u_c) = \Gamma(c)+o_n(1).
\end{align}
As $u_c\in \mathcal{M}(c_1),$ \eqref{E} implies that
$$
E(u_n- u_c)+\Gamma(c_1)\leq \Gamma(c)+o_n(1)
$$
and from the monotonicity assumption on $\Gamma$, i.e. \eqref{lemgammadecrbbb}, we deduce that $E(u_n- u_c) \leq o_n (1)$. On the other hand, we also have
\begin{align} \label{compact1}
\begin{split}
& E(u_n-u_c)-\dfrac{2}{\sigma N}Q(u_n-u_c)\\
&=\gamma\dfrac{\sigma N-4}{2\sigma N}\int_{\R^N}|\Delta (u_n-u_c)|^2\, dx +\dfrac{\sigma N -2}{2\sigma N}\int_{\R^N}|\nabla (u_n-u_c)|^2\, dx,
\end{split}
\end{align}
and since $Q(u_n-u_c)=o_n(1)$ this implies that $E(u_n-u_c) \geq o_n(1)$. Consequently, we have shown that $E(u_n-u_c)=o_n(1)$. As a direct consequence, we conclude from \eqref{E} that $E(u_c) = \Gamma(c)$. When $\sigma N >4$ we also directly deduce from \eqref{compact1} that $\| \Delta(u_n -u_c)\|_{2}=o_n (1)$, $\| \nabla (u_n -u_c)\|_2 =o_n (1)$ 
and 
therefore $\|u_n-u_c\|_{2 \sigma +2}=o_n(1)$ since $Q(u_n-u_c)=o_n(1)$. We merely deduce that $\|\nabla (u_n-u_c)\|_2=o_n(1)$ when $\sigma N =4$ but since, by Lemma \ref{boundcritical}, the sequence $(\| \Delta(u_n -u_c)\|_2)_n$ is bounded, we reach the same conclusion using (\ref{G-N-H1-ineq}) if $N \leq 3$ or (\ref{G-N-H1-ineq2}) if $N \geq 4$. 
\end{proof}
\begin{remark}\label{tooeasy}
Note that if we were able to prove that the inequality in \eqref{lemgammadecrbbb} is strict for any $\bar c \in (c_0,c)$, it would prove that  $u_c\in \mathcal M(c)$. Indeed, in the proof of Lemma \ref{propcompactness}, if we assume that  $c_1<c$ then we reach the contradiction $E(u_c)\ge \Gamma(c_1)>\Gamma(c) = E(u_c)$. Such a strict monotonicity seems however out of reach but nevertheless it is possible to derive the weaker statement that $c \mapsto \Gamma(c)$ is nonincreasing, see 
Lemma \ref{lemgammadecr} in the next section.
\end{remark}
We can now prove our first existence result which basically states the existence of ground states in the range of masses $c > c_0$for which we can prove the positivity of the associated putative multiplier $\alpha_c$.

\begin{proof}[Proof of Theorem \ref{thmmain}]
Fix $c \in (c_0, c_{N, \sigma})$. From  Lemma \ref{psbis} we know that there exists a Palais-Smale sequence $(u_n)_n \subset \mathcal{M}(c)$ for $E$ restricted to $S(c)$ at level $\Gamma(c)$.  By Lemma \ref{conv},  $u_n \rightharpoonup u_c$ in $H^2(\R^N)$, where $u_c$ is a nontrivial solution to
\begin{equation}\label{eq:uc}
\gamma \Delta^2 u_c - \Delta u_c+ \alpha_c u_c = |u_c|^{2 \sigma}u_c 
\end{equation}
for some $\alpha_c \in \R$. Moreover, Lemma \ref{sign-la} implies that $\alpha_c>0$.  In the next section we prove, see Lemma \ref{lemgammadecr}, that $c \mapsto \Gamma(c)$ is nonincreasing.   The combination of Lemma \ref{propcompactness} and Lemma \ref{lemgammadecr}  shows that   
$u_n \to u_c$ in ${L^{2 \sigma +2}(\R^N)}$. 
We conclude from Lemma \ref{conv} that 
 $u_n \to u_c$ in $H^2(\R^N)$. This convergence implies in particular that $E(u_c)= \Gamma(c).$
\end{proof}

\section{Some properties of the function $c \mapsto \Gamma(c)$}\label{behavior}

In this section, we investigate the properties of the map $c\mapsto \Gamma(c)$. All the properties that we establish are also valid for the map $c\mapsto \Gamma_{rad}(c)$ since the arguments can be reproduced when we deal with radially symmetric functions only. Our study is summarized by Theorem \ref{gammmaprop} which is presented at the end of the section.
We begin by showing the continuity of $\Gamma$.

\begin{lem} \label{contprop}
Let $4 \leq \sigma N < 4^*$, then the function $c\mapsto \Gamma(c)$ is continuous on $(c_0, \infty)$.
\end{lem}
\begin{proof}
Let us prove that, for any $c >c_0$, if $(c_n)_n \subset (c_0, \infty)$ is such that $c_n \rightarrow c$ as $n\rightarrow \infty$, then $\lim_{n \to \infty}\Gamma(c_n) = \Gamma(c)$.  
From the definition of $\Gamma(c),$ for any $\eps >0$, there exists $v \in \mathcal{M}(c)$ such that $E(v) \leq \Gamma(c) + \frac{\varepsilon}{2}$. Now, defining
$v_n:=\sqrt{\dfrac{c_n}{c}} \, v \in S(c_n)$,  
we clearly have
$$
\int_{\R^N}|\Delta v_n|^2\, dx \rightarrow \int_{\R^N}|\Delta v|^2\, dx  ,\ \int_{\R^N}|\nabla v_n|^2\, dx\rightarrow \int_{\R^N}|\nabla v|^2\, dx, \mbox{ and} $$
$$ \int_{\R^N}| v_n|^{2\sigma +2}\, dx\rightarrow \int_{\R^N}| v|^{2\sigma +2}\, dx .
$$
In particular, for $n \in \N$ large enough, we get
$$ \gamma \int_{\R^N} | \Delta v_n|^2 \, dx < \frac{N}{N+4} \int_{\R^N}|v_n|^{2 + \frac{8}{N}} \, dx$$
when $\sigma N =4$. Now, using \cite[Lemma 5.2]{BeJeLu} and the above convergences, we deduce that 
\begin{align*}
& \Gamma(c_n) \leq \max_{\lambda >0}E((v_n)_{\lambda})\\
& = \max_{\lambda>0} \left( \frac{{\lambda}^2}{2}{\gamma \int_{\R^N}|\Delta v_n|^2\, dx}+ \frac {\lambda}{2}{ \int_{\R^N}|\nabla v_n|^2\, dx}  - \dfrac{{\lambda}^{\sigma N /2}}{2(2\sigma +2)} { \int_{\R^N}|v_n|^{2\sigma +2}\, dx}\right) \\
&\leq \max_{\lambda>0} \left( \frac{{\lambda}^2}{2}{\gamma \int_{\R^N}|\Delta v|^2\, dx}+ \frac {\lambda}{2}{ \int_{\R^N}|\nabla v|^2\, dx}  - \dfrac{{\lambda}^{\sigma N /2}}{2(2\sigma +2)} { \int_{\R^N}|v|^{2\sigma +2}\, dx}\right) +\dfrac{\varepsilon}{2}\\
&= \max_{\lambda>0}E((v)_{\lambda})+\dfrac{\varepsilon}{2}=E(v)+\dfrac{\varepsilon}{2}\leq \Gamma(c)+\varepsilon.
\end{align*}
This shows that
\begin{equation}\label{limsup}
\limsup_{n \to \infty}\Gamma(c_n) \leq \Gamma(c).
\end{equation}
Next, let $(u_n)_n \subset \mathcal{M}(c_n)$ be such that
\begin{align} \label{inf1}
E(u_n)\leq \Gamma(c_n) +\dfrac{\varepsilon}{2}.
\end{align}
Since $Q(u_n)=0$, using \eqref{limsup} and \eqref{inf1}, we infer that, for $n \in \N$ large enough
\begin{align*}
\gamma\dfrac{\sigma N-4}{2\sigma N}\int_{\R^N}|\Delta u_n|^2\, dx +\dfrac{\sigma N -2}{2\sigma N}\int_{\R^N}|\nabla u_n|^2\, dx  =  E(u_n) \leq
\Gamma(c_n) +\dfrac{\varepsilon}{2} \leq \Gamma(c) +\dfrac{3\varepsilon}{4}
\end{align*}
and thus, when $\sigma N >4$ we immediately deduce that $(u_n)_n \subset H^2(\R^N)$ is bounded. The same conclusion holds true when $\sigma N=4$ by Lemma \ref{boundcritical}. Thus we can assume without loss of generality that 
$$
\int_{\R^N}|\Delta u_n|^2\, dx \rightarrow A, \ \int_{\R^N}|\nabla u_n|^2\, dx\rightarrow B\ \text{ and }\  \int_{\R^N}| u_n|^{2\sigma +2}\, dx\rightarrow  C.
$$
It follows from Remark \ref{lowerboundsinMc} that 
$A>0$ and $C>0$.

Now we define
$\tilde{u}_n:=\sqrt{\dfrac{c}{c_n}} \, u_n \in S(c).$
Using twice \cite[Lemma 5.2]{BeJeLu}, we obtain that 
\begin{align*}
& \Gamma(c) \leq \max_{\lambda >0}E((\tilde{u}_n)_{\lambda})\\
&= \max_{\lambda>0} \left( \frac{{\lambda}^2}{2}{\gamma  \left(\frac{c}{c_n}\right)\int_{\R^N}|\Delta u_n|^2\, dx}+ \frac {\lambda}{2} \left(\frac{c}{c_n}\right){ \int_{\R^N}|\nabla u_n|^2\, dx}  - \dfrac{{\lambda}^{\sigma N /2}}{2(2\sigma +2)}\left(\frac{c}{c_n}\right)^{\sigma +1} { \int_{\R^N}|u_n|^{2\sigma +2}\, dx}\right) \\
&\leq \max_{\lambda>0} \left( \frac{{\lambda}^2}{2}{\gamma A }+ \frac {\lambda}{2} B  - \dfrac{{\lambda}^{\sigma N /2}}{2(2\sigma +2)}  C \right) +\dfrac{\varepsilon}{2}\\
&\leq \max_{\lambda>0} \left( \frac{{\lambda}^2}{2}{\gamma \int_{\R^N}|\Delta u_n|^2\, dx}+ \frac {\lambda}{2}{ \int_{\R^N}|\nabla u_n|^2\, dx}  - \dfrac{{\lambda}^{\sigma N /2}}{2(2\sigma +2)} { \int_{\R^N}|u_n|^{2\sigma +2}\, dx}\right) +\dfrac{3\varepsilon}{4}\\
&= \max_{\lambda>0}E((u_n)_{\lambda})+\dfrac{3\varepsilon}{4}=E(u_n)+\dfrac{3\varepsilon}{4}\leq \Gamma(c_n)+\varepsilon,
\end{align*}
from which we conclude that
\begin{equation}\label{liminf}
\Gamma(c) \leq \liminf_{n \to \infty}\Gamma(c_n).
\end{equation}
The conclusion follows from \eqref{limsup} and \eqref{liminf}.
\end{proof}
We now establish the key monotonicity property that has been used in the proof of Theorem \ref{thmmain}.
\begin{lem}\label{lemgammadecr}
Let $4 \leq \sigma N < 4^*$, then the function $c\mapsto \Gamma(c)$ is nonincreasing on $(c_0, \infty)$.
\end{lem}
\begin{proof}
 To prove the lemma we have to show that if $0 < c_1 < c_2$, then $\Gamma(c_2) \leq \Gamma(c_1)$.
 Noting that, see \eqref{minmax}, 
$$\Gamma(c) = \inf_{u\in S(c)}\sup_{\lambda >0} E(u_\lambda),$$
for any $\varepsilon>0$  there exists a $u_1\in \mathcal{M}(c_1)$ such that
\begin{equation}
\label{fromdef1}
E(u_1)\leq \gamma (c_1)+\dfrac{\varepsilon}{2} \quad \mbox{and} \quad \max_{\lambda>0} E((u_1)_\lambda)=E(u_1)
\end{equation}
where we recall that $(u_1)_{\lambda}(x) := \lambda^{\frac{N}{4}}u_1(\sqrt{\lambda}x)$.
For $\delta>0$, one can find $u_1^\delta \in H^2(\R^N)$ such that $supp\  u_1^{\delta} \subset  B_{\frac{1}{\delta}}(0)$ and $||u_1 - u_1^{\delta}||= o(\delta)$. Thus, as $\delta \to 0$, we have
$$ \int_{\R^N}|\Delta u_1^{\delta}|^2\, dx \rightarrow \int_{\R^N}|\Delta u_1|^2\, dx, \quad  \int_{\R^N}|\nabla u_1^{\delta}|^2\, dx \rightarrow \int_{\R^N}|\nabla u_1|^2\, dx,$$
and $$ \int_{\R^N}|u_1^{\delta}|^{2\sigma +2}\,dx \to \int_{\R^N}|u_1|^{2\sigma +2}\, dx.$$

Now, take $v^\delta\in C^\infty_0 (\R^N)$ such that $supp\ v^\delta \subset B_{\frac{2}{ \delta}+1}(0)\setminus  B_{\frac{2} {\delta}}(0)$, and set
$$
v_0^\delta := (c_2 -\|u_1^\delta \|_2^2)^{\frac 12} \dfrac{v^\delta}{\|v^\delta\|_2}.$$
Define, for $\lambda \in (0,1)$, $w^{\delta}_\lambda :=u_1^\delta +(v_0^\delta)_\lambda .$
Since
$$
\text{dist}(supp\ (v^\delta_0)_{\lambda} , supp\ u_1^\delta )\geq  \frac{1}{\delta}\left(\dfrac{2}{\sqrt{\lambda}}-1\right)>0,
$$
we have that $\|w^{\delta}_\lambda\|_2^2= c_2$. Also, by a standard scaling argument, we see that, as $\lambda $, $\delta \to 0$,
$$
\int_{\R^N}|\Delta w^{\delta}_\lambda|^2\, dx  \rightarrow \int_{\R^N}|\Delta u_1|^2\, dx, \quad 
\int_{\R^N}|\nabla w^{\delta}_\lambda|^2\, dx  \rightarrow \int_{\R^N}|\nabla u_1|^2\, dx, $$
and
$$\int_{\R^N}|w^{\delta}_\lambda|^{2\sigma +2}\,dx   \to \int_{\R^N}|u_1|^{2\sigma +2}\, dx.$$
In \cite[Lemma 5.2]{BeJeLu} it is proved that the function $f : \R^+ \times (\R^+ \cup \{0\}) \times \R^+ \mapsto \R$ defined by $f(a,b,c) := \max_{t>0}(t^2a + tb - c t^{\frac{\sigma N}{2}})$ is continuous if $\sigma N >4$. Setting $(w_{\lambda}^{\delta})_t = t^{\frac{N}{4}}w_{\lambda}^{\delta}(\sqrt{t}x)$, we deduce from the above convergence and (\ref{minmax}) that for $\lambda, \delta>0$ small enough,
$$
\Gamma(c_2) \leq \max_{t>0}E ((w^{\delta}_\lambda)_t )\leq \max_{t>0} E((u_1)_t) +\dfrac{\varepsilon}{2}=E(u_1)+\dfrac{\varepsilon}{2}\leq \Gamma(c_1)+ \varepsilon
$$
and this concludes the proof when $\sigma N >4$. If $\sigma N =4$, observe that for $\lambda, \delta >0$ small enough, we have
$$
\gamma \int_{\R^N}|\Delta w^{\delta}_\lambda|^2 \, dx < \frac{N}{N+4} \int_{\R^N}|w^{\delta}_\lambda|^{2 + \frac 8N} \, dx,
$$
and therefore $\sup_{t>0}E((w_{\lambda}^{\delta})_t) <\infty$. Under this condition, it is straightforward to extend \cite[Lemma 5.2]{BeJeLu} and then we conclude as in the case $ \sigma N >4$.
\end{proof}

\begin{remark}\label{rem:gammaradmonotone}
Since we can reproduce the proof of Lemma \ref{lemgammadecr} when we restrict $\mathcal M(c)$ to radially symmetric functions, i.e. when we consider $\mathcal M_{rad}(c)$, we infer that the map $c\mapsto \gamma_{rad} (c)$ is nonincreasing on $(c_0, \infty)$ under the same assumption $4 \leq \sigma N < 4^*$.
\end{remark}

\begin{lem} \label{monotoneprop}
Let $4 \leq \sigma N < 4^*$ and $c>c_0$. If there exists $u\in S(c)$ such that $E(u) = \Gamma(c)$ and
\begin{equation}\label{eq:monotone}
\gamma \Delta^2 u - \Delta u + \alpha u = |u|^{2 \sigma}u,
\end{equation}
then $\alpha \geq 0$. If moreover $\alpha >0$, then the function $c \mapsto \Gamma(c)$ is strictly decreasing in a right neighborhood of $c$.
\end{lem}
\begin{proof}
Let $u\in S(c)$ be such that $E(u) = \Gamma(c)$ and \eqref{eq:monotone} holds with $\alpha\in\R$. We claim that if $\alpha>0$, respectively $\alpha<0$ , the map $c  \mapsto \Gamma(c)$ is strictly decreasing, respectively strictly increasing, in a right neigbourhood of $c$. 
Let $u_{t,\lambda}(x):=\lambda^{\frac{N}{4}} \sqrt{t} u(\sqrt{\lambda} x)$ for $t,\lambda > 0$. We define
$\beta_E (t,\lambda ):=E(u_{t,\lambda})$, and $\beta_Q (t,\lambda):= Q(u_{t,\lambda})$. We compute
$$
\dfrac{\partial \beta_E}{\partial t}(1,1)=- \dfrac{1}{2}\alpha\, c,\ \,  \dfrac{\partial \beta_E}{\partial \lambda}(1,1)=0,\ \  \dfrac{\partial^2 \beta_E}{\partial \lambda^2}(1,1)<0,
$$
which yields for $|\delta_\lambda|$ small enough and $\delta_t>0$,
\begin{align} \label{lessthan}
\beta_E (1+\delta_t , 1+\delta_\lambda)<\beta_E (1,1), \mbox{ if } \alpha >0,
\end{align}
or
\begin{align} \label{lessthan+}
\beta_E (1-\delta_t , 1+\delta_\lambda)<\beta_E (1,1), \mbox{ if } \alpha <0.
\end{align}
Observe that $\beta_{Q}(1, 1)=0$, and $\dfrac{\partial \beta_Q}{\partial \lambda}(1,1)\neq 0.$
From the Implicit Function Theorem, we deduce the existence of $\varepsilon>0$ and of a continuous function $g :[1-\varepsilon ,1+\varepsilon] \mapsto \R$ satisfying $g(1)=1$ such that $\beta_Q (t,g(t))=0$ for $t\in [1-\varepsilon ,1+\varepsilon]$.
Therefore, we infer from \eqref{lessthan}, if $\alpha>0$, that
$$
\Gamma((1+\varepsilon)c)=\inf_{u\in \mathcal{M}((1+\varepsilon )c)}E(u)\leq E (u_{1+\varepsilon , g(1+\varepsilon)})< E(u_c)=\Gamma(c).
$$
This proves the last assertion of the lemma. 
If $\alpha<0$, we rather conclude from \eqref{lessthan+} that 
$$
\Gamma((1-\varepsilon)c)=\inf_{u\in \mathcal{M}((1-\varepsilon )c)}E(u)\leq E (u_{1-\varepsilon , g(1-\varepsilon)})< E(u_c)=\Gamma(c).$$
Since we know from Lemma \ref{lemgammadecr} that $\gamma$ is nonincreasing, the case $\alpha<0$ is impossible. 

\end{proof}

\begin{cor}\label{gammamonostrict}
Let $N \geq 1$ and $4 \leq \sigma N< 4^*$. 
The map $c \mapsto \Gamma(c)$ is strictly decreasing on $(c_0,c_{N,\sigma})$. 
\end{cor}
\begin{proof}
The proof follows directly by combining Theorem \ref{thmmain} and Lemma \ref{monotoneprop}. 
\end{proof}

\begin{lem}\label{limitprop} 
For all $4 \leq \sigma N < 4^*$, we have $\lim_{c\rightarrow c_0^+}\Gamma(c)= +\infty.$
\end{lem}
\begin{proof}
When $4 < \sigma N < 4^*$, the claim follows directly from  \eqref{relEQ1} and \eqref{lowerboundlap}.
When $\sigma N =4$, to show that $ \lim_{c\rightarrow {c^*_N}^+}\Gamma(c)=\infty$, we need to prove that $\int_{\R^N}|\nabla u_{c_n}|^2 \, dx \to \infty$ for any sequences $(c_n)_n$ with $c_n \to {c^*_N}^+$  and $(u_{c_n})_n \in  \mathcal{M}(c_n)$.
First we observe that for $u \in \mathcal{M}(c)$, using \eqref{poho}, namely that $Q(u)=0$, and \eqref{ajout2},
\begin{align} \label{infty}
\begin{split}
\Gamma(c) \leq E(u)
& \leq \frac {\gamma}{2} \int_{\R^N}|\Delta u|^2 \, dx -\frac{N}{2N+8} \int_{\R^N}|u|^{2 + \frac 8N}\, dx\\
& \leq \frac 12 \left( \left(\frac {c}{c^*_N}\right)^{\frac 4N} -1\right) \gamma \int_{\R^N}|\Delta u|^2 \, dx.
\end{split}
\end{align}
Since $\Gamma(c) >0$, for any $c >c_0$, see Lemma \ref{coercive}, and $c \mapsto \Gamma(c)$ is 
nonincreasing on $(c_0, \infty),$ see Lemma \ref{lemgammadecr}, we deduce the existence of a $\delta >0$ such that,  for any sequences $(c_n)_n \subset \R$ with $c_n \to {c^*_N}^+$  and $(u_{c_n})_n \subset  \mathcal{M}(c_n)$, we have
\begin{equation}\label{ajout3}
0 < \delta \leq \Gamma(c_n) \leq \frac 12 \left( \left(\frac {c_n}{c^*_N}\right)^{\frac 4N} -1\right) \gamma \int_{\R^N}|
\Delta u_{c_n}|^2 \, dx,
\end{equation}
from which we deduce that
\begin{equation}\label{explosion}
\int_{\R^N}|\Delta u_{c_n}|^2 \, dx \to \infty \ \text{as} \ n \to \infty.
\end{equation}
The conclusion now follows from Lemma \ref{boundcritical}. 
\end{proof}

We now investigate the behaviour of the function $c\mapsto \Gamma(c)$ as $c\rightarrow \infty$.
\begin{prop}\label{limitzero}
Let $1 \leq N \leq 3,$ $4 \leq \sigma N < \infty$, and assume in addition that $\frac{4}{3} \leq \sigma <2$ if  $N=3$. Then $\lim_{c \rightarrow \infty}\Gamma(c)=0$.
\end{prop}
\begin{proof}
We first treat the case $\sigma N=4$. 
Using both that $\Gamma(c) = \displaystyle \inf_{u \in S(c)} \sup_{\lambda >0}E(u_{\lambda})$, see (\ref{minmax}), and (\ref{key10})  we get
\begin{align*}
\Gamma(c) \leq \sup_{\lambda >0}E(w_{\lambda})= \frac{\frac{c}{\|U\|_2^2} \left(\int_{\R^N}|\nabla U|^2 \,dx \right)^2}
{8\left(\left(\frac{c}{c_N^*}\right)^{\frac 4N} -1\right) \gamma \int_{\R^N}|\Delta U|^2 \, dx},
\end{align*}
where $w$ is defined by \eqref{defw}. This shows that $\Gamma(c) \to 0$ as $c \to \infty$ since $4/N>1$. 

Assume now $\sigma N>4$ and fix an arbitrary $u \in H^2(\R^N)$ satisfying $\|u\|_2=1$. 
Then $\sqrt{c} u \in S(c)$ and we infer from Lemma \ref{unique} that there exists a unique $\lambda_c^\ast >0$ such that $Q((\sqrt{c} u)_{\lambda_c^\ast}) =0$, i.e.
\begin{align} \label{ajout3}
\lambda_c^\ast \gamma  \int_{\R^N}|\Delta u|^2\, dx+\dfrac{1}{2} \int_{\R^N}|\nabla u|^2\, dx=\frac{\sigma N}{2(2\sigma+2)}\left(c\lambda_c^\ast\right)^{\frac {\sigma N}{2}-1} c^{\sigma +1 - \frac{\sigma N}{2}}\int_{\R^N}|u|^{2\sigma+2}\, dx.
\end{align}
Observe that $\frac{\sigma N}{2}-1 >0$. Also, recording that $\sigma<2$ when $N=3$, we have that $1 + \sigma - \frac{\sigma N}{2}>0.$ We then deduce from \eqref{ajout3} that $c\lambda_c^\ast \to 0$ as $c \to \infty$. 
Now, using again (\ref{minmax}) and (\ref{key10}),  it follows that 
$$
\Gamma(c) \leq E((\sqrt{c}u)_{\lambda_c^\ast})
= c (\lambda_c^\ast )^2\gamma\dfrac{\sigma N-4}{2\sigma N}\int_{\R^N}|\Delta u|^2\, dx + c \lambda_c^\ast \dfrac{\sigma N -2}{2\sigma N}\int_{\R^N}|\nabla u|^2 \, dx,
$$
and we conclude that $\Gamma(c) \to 0$ as $c \to \infty$.
\end{proof}

To study the cases $\sigma > 2$ when $N=3$,  $\sigma > 1$ when $N=4$ and 
$ 4 \leq \sigma N < 4^*$ when $N \geq 5$,
we consider first the following equation
\begin{equation}\label{4NLSalpha=0}
\gamma\Delta^{2}u - \Delta u= |u|^{2\sigma}u,
\end{equation}
which appears as a limit equation when $c\to\infty$ as we later show.
The natural associated energy space is defined by 
\begin{equation}\label{defX}
X:=\{u \in D^{1,2}(\R^N): \int_{\R^N}|\Delta u|^2 \, dx < \infty\}
\end{equation} 
that we equip with the norm
$$
\|u\|_X^2:=\int_{\R^N} |\Delta u|^2\, dx + \int_{\R^N} |\nabla u|^2\, dx.
$$
Assuming $\sigma > 2$ when $N=3$, $\sigma > 1$ when $N=4$, and 
$ 4 \leq \sigma N < 4^*$ when $N \geq 5$, we see from (\ref{G-N-H1-ineq2}) that  $X$ embeds continuously in $L^{2 \sigma + 2}(\R^N)$ and therefore the functional $E$ defined by \eqref{def:E} is well-defined in $X$ and we will show that it has critical points inside that space. On the other hand, using Pohozaev identity and multiplying \eqref{4NLSalpha=0} by $u$ and integrating, we get
$$ 0= \gamma \left(\dfrac{4\sigma +4}{\sigma N}-1 \right)\int_{\R^N}|\Delta u|^2\, dx+ \left(\dfrac{2\sigma +2}{\sigma N} -1\right)\int_{\R^N}|\nabla u|^2\, dx.$$
This shows \eqref{4NLSalpha=0} has no solution if $4-\sigma (N-4)$ and $2-\sigma (N-2)$ have the same sign. Therefore \eqref{4NLSalpha=0} has no solution if
$\sigma \leq 2$ if $N=3$, $\sigma \leq 1$ if $N=4$ and $\sigma N \leq 2^\ast$ or $\sigma N \ge 4^\ast$ when $N\geq 5$.

Now, define
$$m := \inf\{ E(u) | u \in X \backslash \{0\}, dE(u) =0\}.$$%
\begin{prop} \label{minimizer}
Assume that $\sigma > 2$ if $N=3$, $\sigma > 1$ if $N=4$ and $2^\ast < \sigma N < 4^*$ if $N \geq 5$. Then $m >0$ and is achieved. Moreover,
\begin{enumerate}
\item if $N=3,4$, no minimizer belongs to $L^2(\R^N)$;
\item if $N \geq 5$, all minimizers belong to $L^2(\R^N)$.
\end{enumerate}
\end{prop}

\begin{proof}
It is classical to show that  $m$  is achieved with $m >0$  if and only if
$$J(u) = \int_{\R^N} \gamma |\Delta u|^2 + |\nabla u|^2 \, dx  \quad \mbox{admits a minimizer on} \quad  M:= \{ u \in X : ||u||_{2 \sigma +2}=1\}.$$%
 To prove the existence of a minimizer, we proceed as in \cite[Remark 3.2]{BoNa}. Let $(u_n)_n \subset M$ be a minimizing sequence. Without loss of generality, since $H^2(\R^N)$ is dense in $X$, we can assume that $(u_n)_n \subset H^2(\R^N)$. Then we set $f_n = - \sqrt{\gamma}\Delta u_n + \frac{u_n}{2\sqrt{\gamma}}$ and define $v_n \in H^2(\R^N)$ to be the strong solution of $- \sqrt\gamma\Delta v_n + \frac{v_n}{2\sqrt\gamma}= |f_n|^* $ in $\R^N$ where $|f_n|^*$ denotes the Schwarz symmetrization of $|f_n|$. Thus for each $n \in \N$ we have $v_n \in H^2_{rad}(\R^N)$ and a particular case of \cite[Lemma 3.4]{BoMoRa} implies that
\begin{align*}
\begin{split}
J\left(\frac{v_n}{||v_n||_{2 \sigma +2} } \right)
& =  \frac{\int_{\R^N}(- \sqrt{\gamma} \Delta v_n + \frac{v_n}{2\sqrt{\gamma} })^2 \, dx - \frac{1}{4\gamma} \int_{\R^N}v_n^2 \, dx }{||v_n||_{2 \sigma +2}} \\
& \leq \frac{\int_{\R^N}(- \sqrt{\gamma} \Delta u_n + \frac{u_n}{2\sqrt{\gamma}})^2 \, dx - \frac{1}{4\gamma} \int_{\R^N}u_n^2 \, dx }{||u_n||_{2 \sigma +2}}  = J\left(\frac{u_n}{||u_n||_{2 \sigma +2}} \right).
\end{split}
\end{align*}

This shows that $(\tilde{v}_n)_n := \left( \frac{v_n}{||v_n||_{2 \sigma +2}}\right)_n$ is again a minimizing sequence. Now we claim that $X_{rad}$, the subset of radially symmetric functions in $X$, is compactly embedded into $L^{2 \sigma +2}(\R^N)$. Indeed, it is well-known that if $u\in  D^{1,2} (\R^N)$ is radially symmetric, we have (see for instance \cite[Radial Lemma A.II]{BeLi1}), for $|x|\neq 0$,
$$|u(x)| \leq C |x|^{-(N-2)/2} \|\nabla u\|_2.$$
Using this pointwise decay, we get 
\begin{align*}
\int_{\R^N \backslash B_R(0)} |u|^{2 \sigma +2} dx&\leq \int_{\R^N \backslash B_R} |u|^{2 \sigma +2  -2N/(N-2) }  |u|^{2N/(N-2) } dx\\
&\leq C  R^{- \frac{N-2}{2} (2 \sigma +2 - \frac{2N}{N-2} )} \|\nabla u\|_2^{( N-2)/ N}.  
\end{align*}
Since we also have local compactness, the claim follows. Using this embedding, we get that $(\tilde{v}_n)_n$ weakly converges to some $v \in X$ with $||v||_{2 \sigma +2}=1$ and the remaining arguments are standard. 

We now prove that $m$  does not have a minimizer in $H^2(\R^N)$ when $N=3,4$.  Assuming by contradiction that $u$ is such a minimizer, we deduce from \cite[Lemma 3]{BoNa} that $u$ must have a sign and without loss of generality, we can assume that $u \geq 0$. To conclude it is therefore enough to show that (\ref{4NLSalpha=0}) has no nontrivial nonnegative solutions in $H^2(\R^N)$. To this aim, we 
decompose \eqref{4NLSalpha=0} into the following elliptic system
\begin{align} \label{system}
     \begin{cases}
     & - \gamma \Delta u =v,\\
     & - \Delta v +  \displaystyle \frac{1}{\gamma} v=|u|^{2\sigma +2} u.
     \end{cases}
\end{align}
The weak maximum principle applied to the second equation in \eqref{system} shows that $v\geq 0$ and thus any nontrivial nonnegative solution $u$ to \eqref{4NLSalpha=0} satisfies $-\Delta u \geq 0$. The Liouville type result of \cite[Lemma $4.2$]{MoVS} now show that
$u \notin L^2(\R^N)$.  
The proof of the assertion (2) is postponed to Proposition \ref{finitemass} in the Appendix.
\end{proof}

\begin{remark}
If one considers \eqref{4NLSalpha=0} with $N=1,2$ or $N=3$ and $\sigma \leq 2$ or $N=4$ and $\sigma =1$, Lemma \ref{sign-la} directly shows that there is no solution in $H^2(\R^N)$ nor in $X$ (as shown by Remark \ref{rem:pohoX}). 
\end{remark}

Note that by standard arguments, $m$ can also be defined as
\begin{equation}\label{POhozaev}
m := \inf \{ E(u) | u \in X \backslash \{0\}, Q(u) =0\}.
\end{equation}%

\begin{prop}\label{limitpositive}
Assume $\sigma > 2$ when $N=3$, $\sigma > 1$ when $N=4$, and 
$ 4 \leq \sigma N < 4^*$ when $N \geq 5$.
Then $$\lim_{c \rightarrow \infty} \Gamma(c)= m >0.$$
\end{prop}

\begin{proof} Using the definition (\ref{POhozaev}), we directly obtain that $\Gamma(c) \geq m$ for all $c >c_0$. Now, still from (\ref{POhozaev}) and taking Proposition \ref{minimizer} into account, we infer that there exists $u \in X\setminus L^2(\R^N)$ such that $E(u) = m$ and $Q(u)=0$. For $R>0$, we define
 $u_R(x):=\eta(\frac{x}{R})u(x)$, where $\eta(x)=1$ for $|x| \leq 1$, $\eta(x)=0$ for $|x| \geq 2$, and $0 \leq \eta \leq 1$. Thus, as $R \to \infty$, we have
$$c_R:=\|u_R\|_2 \to \infty, \|u_R\|_{2 \sigma +2} \to \|u\|_{2 \sigma +2}, \|\nabla u_R\|_{2} \to \|\nabla u\|_{2} \mbox{ and } \| \Delta u_R\|_{2} \to \|\Delta u\|_{2}.$$

Now, let  $\lambda_R^{\ast} >0$ be such that $Q((u_R)_{\lambda_R^{\ast}})=0$. By uniqueness of $\lambda_R^{\ast}$, see Lemma \ref{unique}, we obtain that $\lambda_R^{\ast} \to 1$ as $R \to \infty$. Thereby
$$
\Gamma(c_R) \leq E((u_R)_{\lambda_R}) \leq E(u) + o_R(1) = \Gamma(\infty) + o_R(1),
$$
where $o_R(1) \to 0$ as $R \to \infty$. This means $\lim_{R\to\infty}\Gamma(c_R)=m$. As from Lemma \ref{lemgammadecr} we know that $\gamma$ is nonincreasing, $\gamma$ has a limit at infinity and therefore $\lim_{c\to\infty}\Gamma(c) = m$. 
\end{proof}

In view of Proposition \ref{limitpositive} and for simplicity of notation, we define $\Gamma(\infty):= m = \lim_{c \to \infty}\Gamma(c).$ \medskip

It is clear from Corollary \ref{gammamonostrict} that if $4 \leq \sigma N< 4^*$ and $c_{N,\sigma}=\infty$, then $\Gamma(c) > \Gamma(\infty)$ for all $c > c_0$. This is still true for every $\sigma \geq 2$ if $N=3$ and every $\sigma \geq 1$ if $N=4$.
\begin{cor}\label{above}
Let $\sigma > 2$ if $N=3$ or $\sigma > 1$ if $N=4$. Then $\Gamma(c) > \Gamma(\infty)$ for all $c > c_0$. 
\end{cor}
\begin{proof}
Assume by contradiction that there exists  $c>0$ such that 
$\Gamma(c) = \Gamma(\infty)$. Let $(u_n)_n \subset  \mathcal{M}(c)$ be a Palais-Smale sequence for $E$ restricted to $ \mathcal{M}(c)$ at level $\Gamma(c)$. We know from Lemma \ref{lemgammadecr} that $\gamma$ is nonincreasing. Therefore we can apply Lemma \ref{conv} and Lemma \ref{propcompactness}. We then deduce the existence of $u_c \in H^2(\R^N)$ satisfying \eqref{eq:monotone} for some $\alpha\in\R$,  $0 < c_1:=||u_c||_2^2 \leq c$ and 
$E(u_c) = \Gamma(c)$. By definition and monotonicity of $\gamma$, we deduce that $\Gamma(c)\le \Gamma(c_1)\le E(u_c) = \Gamma(c) = \Gamma(\infty).$
Thus $u_c$ is a ground state on the constraint $S(c_1)$ and necessarily it is a solution of \eqref{4NLSalpha=0} by Lemma \ref{monotoneprop}. Since $E(u_c) = \Gamma(\infty)$, the assertion (1) of Proposition \ref{minimizer} implies $u_c\not\in L^2(\R^N)$. This contradiction ends the proof.
\end{proof}

\begin{prop}\label{bigN}
Assume $4 \leq \sigma N < 4^*$ and $N \geq 5$. There exists $c_{\infty} >0$ such that $\Gamma(c) = \Gamma(\infty)$ for all $c \geq c_{\infty}$.
\end{prop}
\begin{proof}
We know from Proposition \ref{minimizer} and the characterization \eqref{POhozaev} that there exists $u \in H^2(\R^N)$ such that $E(u) = \gamma (\infty)$ and $Q(u)=0$. Set $c_{\infty}:= ||u||_2^2$. Obviously, we have $\Gamma(c_{\infty}) = \Gamma(\infty)$. Since $\gamma$ is nonincreasing, see Lemma \ref{lemgammadecr}, and its limit at infinity is $\Gamma(\infty)$, see Proposition \ref{limitpositive}, we conclude that $\gamma$ is constant for $c\ge c_{\infty}$.  
\end{proof}
As a consequence of the results established in this section we can state the following theorem.
\begin{thm} \label{gammmaprop}
Let $4 \leq \sigma N < 4^*.$ For any $c >c_0$ the function $c \mapsto \Gamma(c)$ is continuous, nonincreasing and $\lim_{c\rightarrow c_{0}^+}\Gamma(c)=\infty.$ In addition
\begin{enumerate}[(i)]
\item If $N= 1,2$, $N=3$ with $\frac{4}{3}\leq \sigma \leq 2$ or $N=4$ with $\sigma =1$, then $c \mapsto \Gamma(c)$ is strictly decreasing. Moreover, $\lim_{c \rightarrow \infty}\Gamma(c)=0$, when $N= 1,2$ or $N=3$ with $\frac{4}{3}\leq \sigma < 2$.
\item  If $N=3$ with $\sigma > 2$ or $N=4$ with $\sigma >1$, then  $\lim_{c \rightarrow \infty}\Gamma(c) := \Gamma(\infty) >0$ and $\Gamma(c) > \Gamma(\infty)$ for all $c >c_0.$ 
\item  If $N \geq 5$, then $\lim_{c \rightarrow \infty}\Gamma(c) := \Gamma(\infty) >0$ and there exists $c_{\infty} > c_0$ such that $\Gamma(c) = \Gamma(\infty)$ for all $c \geq c_{\infty}.$ 
\end{enumerate}
\end{thm}

\begin{proof}[Proof of Theorem \ref{gammmaprop}]
The proof follows directly from Propositions \ref{limitzero}, \ref{limitpositive}, \ref{bigN}  and from Corollaries \ref{gammamonostrict} and \ref{above}. 
\end{proof}

We leave as open question the study of $\lim_{c \rightarrow \infty}\Gamma(c)$ when $N=3$, $\sigma=2$ and $N=4$, $\sigma=1$. We conjecture that the limit is zero in those cases.

\section{Radial solutions, proof of Proposition \ref{sharpdecay} and of Theorems \ref{thm} and \ref{radN34}}\label{multiplicity}

In this section, we focus on radial solutions. We begin by giving the proof of Theorem \ref{thm}. We denote by $\sigma : H^2 (\R^N)\rightarrow H^2 (\R^N)$ the transformation $\sigma (u)=-u$. The following definition is \cite[Definition 7.1]{Ghoussoub}.
\begin{definition}\label{Ghoussoub-71}
Let $B$ be a closed $\sigma$-invariant subset of $Y \subset H^2_{rad}(\R^N )$. We say that a class $\mathcal{F}$ of compact subsets of $Y$ is a $\sigma$-homotopy stable family with closed boundary $B$ if
\begin{enumerate}
\item every set in $\mathcal{F}$ is $\sigma$-invariant.
\item every set in $\mathcal{F}$ contains $B$;
\item for any $A\in \mathcal{F}$ and any $\eta \in C([0,1]\times Y, Y)$ satisfying, for all $t\in [0,1]$, $\eta (t,u)=\eta (t,\sigma (u))$,  $\eta (t,x)=x$ for all $(t,x)\in (\{0\}\times Y)\cup ([0,1]\times B)$, we have $\eta (\{1\} \times A)\in 
\mathcal{F}$.
\end{enumerate}
\end{definition}
\begin{lem}
\label{psbis2}
Let $\mathcal{F}$ be a $\sigma$-homotopy stable family of compact subsets of $\mathcal{M}(c)$ with a close boundary $B$. Let $c_{\mathcal{F}}:= \inf_{A\in \mathcal{F}}\max_{u\in A}E(u)$.  Suppose that $B$ is contained in a connected component of $\mathcal{M}(c)$ and that $\max \{\sup E(B),0\}<c_{\mathcal{F}}<\infty$. Then there exists a Palais-Smale sequence $(u_n)_n \subset \mathcal{M}(c)$ for $E$ restricted to $S(c)$ at level $c_{\mathcal{F}}$.
\end{lem}
\begin{proof}
We are brief here and refer to \cite{BaSo2} for a proof of a closely related result, \cite[Theorem 3.2]{BaSo2}. The proof of Lemma \ref{psbis2} first relies on an equivariant version of Lemma \ref{ps} whose proof makes use of \cite[Theorem 7.2]{Ghoussoub} instead of \cite[Theorem 3.2]{Ghoussoub} but which for the rest is almost identical. Then the lemma follows just as \cite[Theorem 3.2]{BaSo2} follows from \cite[Proposition 3.9]{BaSo2}.
\end{proof}

Next, we recall the definition of the genus of a set due to M.A. Krasnosel'skii, adapted to our setting,
\begin{definition}\label{genus}
For any closed $\sigma$-invariant set $A \subset H^2(\R^N)$, the genus of $A$ is defined by
$$
j(A):= \min \{n \in \N^+ : \exists \ \varphi : A \rightarrow \R^n\backslash \{0\}, \varphi \,\, \text{is continuous and odd}\}.
$$
When there is no $\varphi$ as described above, we set $j(A)= \infty.$
\end{definition}
Now let $\mathcal{A}_\mathcal{F}$ be the family of compact and $\sigma$-invariant sets  contained in 
$\mathcal{M}_{rad}(c)$.  For any $k \in \N^+$, define
$$
J_{k}:=\{A \in \mathcal{A}_\mathcal{F}: j(A) \geq k\}
$$
and
$$
\beta_k:=\inf_{A \in \Gamma_{k}}\sup_{u \in A}E (u).
$$

\begin{lem} \label{nonempty} $ $
\begin{enumerate}[(i)]
\item Let $4 <\sigma N < 4^*$, then for  any $k \in \N^+$, $J_k \neq \emptyset$.
\item  Let $\sigma N =4$, then for any $k \in \N^+$ there exists a $c(k) > c_{N}^*$ such that $J_k \neq \emptyset$ for all $c \geq c(k).$
\end{enumerate}
\end{lem}
\begin{proof}
For $k \in \N^+$ and $V \subset H^2_{rad}(\R^N)$ such that $\dim V=k$, we set $SV(c):=V \cap S(c)$.
First we consider the case $4 <\sigma N < 4^*$. 
By the basic property of the genus, see \cite[Theorem 10.5]{AmMa}, we have that $j(SV(c))= \dim V =k$. In view of Lemma \ref{unique}, for any $u\in SV(c)$ there exists a unique $\lambda_u^\ast >0$ such that $ u_{\lambda_u^\ast} \in \mathcal{M}(c)$. It is easy to check that the mapping $\varphi: SV(c)\rightarrow \mathcal{M}(c)$ defined by $\varphi (u)=u_{\lambda_u^\ast}$ is continuous and odd. Then \cite[Lemma 10.4]{AmMa} leads to $j( \varphi(SV(c))) \geq j(SV(c))=k$ and this shows that $J_k \neq \emptyset$.  

In the case $\sigma N=4$ we prove that $J_k \neq \emptyset$ when $c > c_N^*$ is large enough. 
Using the fact that all norms are equivalent in a finite dimensional subspace, we get, for $c>c_N^*$ large enough and for any $u \in SV(c)$,
$$
\gamma\int_{\R^N}|\Delta u|^2 \, dx < \frac{N}{N+4} \int_{\R^N}|u|^{2 + \frac 8N} \, dx.
$$
This shows that $\sup_{\lambda >0}E(u_{\lambda}) < \infty$ and thus we can apply Lemma \ref{unique} for any $u \in SV(c)$. Namely that there exists a unique $\lambda_u^{\ast}>0$ such that $Q(u_{\lambda_u^{\ast}})=0$. At this point, we conclude as in the case $4 < \sigma N < 4^*$.
\end{proof}

\begin{proof}[Proof of Theorem \ref{thm}]
Consider the minimax level $\beta_k$. From Lemma \ref{nonempty} we know that each of the classes $J_k$ is non empty and thus to each of them we can apply Lemma \ref{psbis2} to obtain the existence of  Palais-Smale sequences $(u_n^k)_n \subset \mathcal{M}_{rad}(c)$ for $E$ restricted to $S(c)$ at the levels $\beta_k$. By Lemma  \ref{conv} we know that, up to a subsequence, $(u_n^k)_n$ converges strongly in $H^2(\R^N)$ if $(u_n^k)_n$ converges strongly in $L^{2 \sigma +2}(\R^N)$ and if the limit $u^k$ solves \eqref{eq:uc} for some $\alpha_c^k >0$. The first condition holds because the embedding 
$H^2_{rad}(\R^N)\hookrightarrow L^p(\R^N)$ is compact for $2< p< 4^*$ whereas the positivity of $\alpha_c^k$ follows from Lemma \ref{sign-la} when $ c<c_{N,\sigma}$. Observe that in the case $\sigma N=4$,  we can conclude only if $c(k) \leq c_{N,\sigma}$. This last inequality holds true when $2\leq N\leq 4$ since $c_{N,\sigma}=\infty$. Thus, under our assumptions,  $(u_n^k)_n$ converges to a $u^k$ which is a critical point of $E$ on $S(c)$.  Now to show that if two (or more) values of $\beta_k$ coincide, then $E$ has infinitely many critical points at level $c(k)$, one can either proceed in the usual way, or adapt \cite[Lemma 6.4]{BaSo2} to the present setting.
\end{proof}

Next, we turn to the proof of Theorem \ref{radN34}. Compared with Theorem \ref{thm}, the additional argument we use to eliminate the restriction $c\in (0, {c}_{\sigma, N})$ is a sharp decay estimate for the solutions of \eqref{4NLSeqend}. 
Before proceeding, we recall some classical facts. Let $G$ be the fundamental solution to $\Delta^2 -\Delta$, i.e.
$$\Delta^2 G -\Delta G = \delta_0.$$
We recall that
$$G= g_{-1} - g_0,$$
where $g_{-a}$ is the fundamental solution to $-\Delta +a $, where $a\in \R $, see \cite[Proposition 3.13]{BoCaMoNa}. It is well-known that $g_0 (x)= c_N |x|^{2-N}$ whereas for $a>0$,
$$g_{-a} (x)=\begin{cases}C_1 |x|^{2-N} + o(|x|^{2-N}),\ when\ |x|\rightarrow 0^+, \\  
C_2\dfrac{e^{-\sqrt{a} |x|}}{|x|^{\frac{N-1}{2}}} + o\left( \dfrac{e^{-\sqrt{a} |x|}}{|x|^{\frac{N-1}{2}}}\right),\ when\ |x|\rightarrow \infty,\end{cases}$$
where $C_1,C_2>0$.

\begin{proof}[Proof of Proposition \ref{sharpdecay}]
To simplify the notations, we assume that $\gamma=1$. By elliptic regularity, $u\in C^{4,\alpha}$, see for instance \cite[Theorem 3.7]{BoCaMoNa}.

\medbreak 

\noindent{\bf Claim 1 :} $|u|^{2\sigma}u \in L^{1}(\R^N) $. 
Since $u\in X$, $u\in L^{s}(\R^N)$ for $s\ge \frac{2N}{N-2}$ and therefore we can assume that $2\sigma +1< \frac{2N}{N-2}$ otherwise we are done. Let $t=2\sigma +2$. Since $u^{t-1}\in L^{\frac{2N}{(N-2)(t-1)}}(\R^N)$, $W^{2,p}$-estimates imply that $-\Delta u \in L^{\frac{2N}{(N-2)(t-1)}}(\R^N)$ and in turn that $u\in L^{\frac{2N}{(N-2)(t-1)-4}}(\R^N)$ by  
Hardy-Littlewood-Sobolev inequality, see \cite[Chapter 5, \S 1.2 - Theorem 1]{Stein}. Since $N\leq 6$, we have 
$$q_0=\dfrac{2N}{(N-2)(t-1)-4}< \dfrac{2N}{N-2}.$$ 
If $q_0\leq t-1$, we are done. Otherwise we proceed by induction. Setting 
$$q_{k+1}=\dfrac{q_k N}{(t-1)N-2 q_k}$$
 for $k\ge 0$, we deduce in the same way that $\Delta u\in L^{\frac{q_k}{t-1}}(\R^N)$ implies $u\in L^{q_{k+1}}(\R^N)$ as long as $q_k>t-1$.
Moreover, if $q_k>t-1$, we have $0<(t-1)N - 2q_k < (t-1)(N-2)$ which yields 
$$
q_k-q_{k+1}
= q_k \frac{(t-2) N - 2q_k}{(t-1)N - 2q_k}
> \dfrac{(t-2) N - 2q_k}{N-2}
$$
and using the fact that $q_k<t$ implies $(t-2) N - 2q_k>t(N -2)-2N$, we infer that 
$$q_k-q_{k+1}> \dfrac{t(N -2)-2N}{N-2}=t-\dfrac{2N}{N-2}.$$
Since $\sigma N > 2^\ast$, we have that $\varepsilon = t-\dfrac{2N}{N-2}>0$ and therefore $q_{k+1}<q_k-\varepsilon$. 
It follows that $q_k<t-1$ after a finite number of iterations and the claim is proved.  

\medbreak

\noindent{\bf Claim 2 :} there exists a constant $C>0$ such that, for $|x|$ large enough, 
$$|u(x)|\leq C|x|^{2-N}.$$ 
By Claim 1, we know $f=|u|^{2\sigma} u \in L^1 (\R^N) \cap L^\infty (\R^N)$. Thus,  we deduce from \cite[Proposition $3.14$]{BoCaMoNa} that 
\begin{equation}
\label{ewrep}
u(x)= \int_{\R^N} f(y) G(x-y)dy.
\end{equation}
Next, we follow closely \cite[Lemma 2.9]{EvWe}. Let $R>0$. We define $M_R=\R^N \backslash B_R$ and $\tilde{f}=|u|^{2\sigma}$. Using H\" older inequality with $q=1+\frac{1}{2\sigma}$ and noticing that $|G(x)|\leq C |x|^{2-N}$, we have
\begin{align}
\label{odehe4}
\int_{M_R} & |\tilde{f}(y)| |G(x-y)| dy \leq  C \sup_{y\in M_R}|\tilde{f}(y)| \left(\int_{B_1 (x)} |x-y|^{ (2-N)} dy \right)\\
& + C \left(\int_{M_R}|\tilde{f}(y)|^q dy \right)^{1/q} \left(\int_{\R^N \backslash B_1 (x)} |x-y|^{(2-N) q^\prime } dy \right)^{1/q^\prime}\rightarrow 0,\ as\ R\rightarrow \infty .\nonumber
\end{align}
Notice that $(2-N)q^\prime <-N$ thanks to our assumptions on $\sigma$. So we choose $R>0$ such that
\begin{equation}
\label{odehe1}
 \sup_{x\in \R^N} \int_{M_R} |\tilde{f}(y)| |G(x-y)|dy < 2^{-N}.
\end{equation}
We define
 $$u_0 (x)=\int_{B_R} G(x-y) \tilde{f}(y) u(y) dy,\ B_0(x)=\int_{M_R}G(x-y) \tilde{f}(y) u(y) dy,$$
and
$$u_k (x)=\int_{M_R}G(x-y) \tilde{f}(y) u_{k-1}(y) dy,\ B_k (x)=\int_{M_R}G(x-y) \tilde{f}(y) B_{k-1}(y) dy. $$
Using the fact that $u_{k+1} =B_{k}-B_{k+1}$, for all $k\geq 0$, we see that for any $m\in \N$, we can decompose $u$ as 
$$u=\sum_{k=0}^m u_k +B_m.$$
Set $\beta_k =\sup_{|x|\geq R}|B_k (x)|$.
One can show proceeding as above and using that $\tilde{f}u \in L^1 (\R^N) \cap L^s (\R^N)$ that $\beta_0 <\infty$. Using \eqref{odehe1}, we find that $\beta_k \leq 2^{-N}\beta_{k-1}\leq 2^{-kN}\beta_0\rightarrow 0$ as $k\rightarrow \infty$. This implies that $B_m \rightarrow 0$ as $m \rightarrow \infty$ which shows that $u=\sum_{k=0}^\infty u_k $ holds uniformly in $M_R$. Set $\mu_k =\sup_{|x|\geq R} |x|^{N-2}|u_k (x)|$. 
Since $u_0\in L^\infty (\R^N )$ and, for $|x|\geq 2R$,
$$|u_0 (x)|\leq  C \int_{B_R} |x-y|^{2-N} |\tilde{f} (y) u(y)|dy \leq C |x|^{2-N},$$
we have $\mu_0<\infty $. Moreover, we have, using Young inequality and \eqref{odehe1},
\begin{align*}
|x|^{N-2} |u_k (x)|&\leq C \mu_{k-1} |x|^{N-2}\int_{M_R} |\tilde{f}(y)||y|^{\textcolor{red}{2-N}}|G(x-y )| dy \\
&\leq C \mu_{k-1} \int_{M_R} |\tilde{f} (y)| (1+ |y|^{2-N} |x-y|^{N-2}) |x-y|^{2-N}\\
&\leq \frac{\mu_{k-1}}{2}.
\end{align*}
Iterating the previous estimate, we find that $\sum_{k\geq 0}\mu_k <\infty$. Since $u=\sum_{k=0}^\infty u_k$, we deduce that $|u(x)|\leq C |x|^{2-N}$, for $|x| $ large enough.
\medbreak

\noindent {\bf Conclusion.} To derive the sharp decay, we use a maximum principle for linear cooperative systems. Letting $\tilde{f}(x)=|u|^{2\sigma}$, we can rewrite \eqref{4NLSeqend} as
\begin{equation}
\label{prop5.1sys}
\begin{cases}-\Delta u=v,\\ -\Delta v+v= \tilde{f} u .\end{cases}
\end{equation}
Applying \cite[Theorem 3]{Sir} with $L= \begin{pmatrix} \Delta \\ \Delta \end{pmatrix}$ and $C=\begin{pmatrix}0 & 1 \\ \tilde{f}  & -1  \end{pmatrix}$, we see that if there exists a $C^2$ function $V$ such that  $V(x)\ge 0$ and $-\Delta V(x)\ge 0$, for all $x\in \partial{\Omega}$ and 
\begin{equation}
\label{supsolfin}
\Delta^2 V -\Delta V - \tilde f(x) V\geq 0,\ in\ \Omega ,
\end{equation}
then the maximum principle applies for the vector $U=\begin{pmatrix} u\\ v  \end{pmatrix}$ i.e. if $LU+CU\leq 0$ in $\Omega$ and $U\geq 0$ in $\partial \Omega$, then $U\geq 0$ in $\Omega$. We claim that $u$ has a constant sign in $\R^N \backslash B_{R}$ if $R>0$ is large enough. Indeed, suppose by contradiction that this is not the case. Then $u$ has to oscillate around $0$ which implies that there exist $R_2>R_1>R$ such that $ u$ attains a positive local maximum at $R_1$ and $R_2$ for which $U(R_i) \geq 0$, $i=1,2$, and a negative minimum in between. From Claim 2, we deduce that $\tilde f(x)\leq C |x|^{2\sigma (2-N)}$, for $|x|$ large enough. So taking $\Omega =B_{R_2} \backslash B_{R_1}$, one sees that for $\varepsilon>0$ small and $R_1>R_2>R	$ big enough, the function $V(x)=|x|^{2-N+\varepsilon}$ satisfies \eqref{supsolfin} since 
$$
\Delta^2 V(x)-\Delta V(x) - \tilde f(x)V(x)
  \geq  \varepsilon (N-2-\varepsilon)|x|^{-N+\varepsilon}\left(1+(-N+\varepsilon)(-2+\varepsilon)|x|^{-2} - C |x|^{2-2\sigma (N-2)}\right)
$$
and $2-2\sigma (N-2)<0$. Thus the maximum principle applies and gives that $u$ and $v=-\Delta u \geq 0$ in $B_{R_2} \backslash B_{R_1}$ which yields a contradiction. Therefore, there exists $R>0$ such that $u$ has a constant sign in $\R^N \backslash B_{R}$. We assume without loss of generality that $u\geq 0$ in $\R^N \backslash B_{R}$. Hence, applying the maximum principle to the second equation of \eqref{prop5.1sys}, we deduce that $v=-\Delta u \geq 0$ in $\R^N \backslash B_{R}$. Thanks to \cite[Lemma 4.2]{MoVS}, we deduce that there exists $C\in \R\setminus\{0\}$ such that $$\lim_{|x|\to\infty}\frac{u(x)}{|x|^{2-N}}=C.$$ 
\end{proof}
\begin{remark}
\begin{enumerate}[(i)] 
\item Proposition \ref{sharpdecay} implies in particular that, under its assumption, any radial solution $u\in X$ to \eqref{4NLSeqend} does not belong to $L^2 (\R^N)$.

\item  
Let $\alpha> 0$ and assume that $N=3$ and $\sigma > 2$ or $N=4$ and $\sigma > 1$. Assume that $u\in L^{2\sigma +2}$ is a radial solution to 
\begin{equation*}
 \gamma \Delta^2 u - \Delta u  -\alpha u = |u|^{2 \sigma}u.
\end{equation*}
We conjecture that also in this case, the solution $u$ does not belong to $L^2(\R^N)$ (see \cite{GLZ,Mandel} for the case $\gamma=0$). In this direction, we only established in \cite{BCM} that $|u(x)|\leq C |x|^{\frac{1-N}{2}}$, for some constant $C>0$. 
\end{enumerate}
\end{remark}

The next statement is a direct corollary of Proposition \ref{sharpdecay}, Lemma \ref{monotoneprop} and Lemma \ref{sign-la}, both extended to the radial setting.
\begin{cor} \label{sign-la2}
Let $N=3,4$ and assume that $4 < \sigma N < \infty $. For any $c>0$, if  $u_c \in S(c)$ is a radial solution to 
\begin{equation}\label{eq:min-alphac}
\gamma \Delta^2 u - \Delta u + \alpha u = |u|^{2 \sigma}u,
\end{equation}
with $E(u_c) = \gamma_{rad} (c)$, then $ \alpha> 0$. 
\end{cor}

\begin{proof}[Proof of Theorem \ref{radN34}]
For $N=1,2$ and $\sigma N\geq 4$ or $N=3$ and $4/3 \leq \sigma \leq 2$ or $N=4$ and $\sigma=1$, Theorem \ref{thmmain} can just be recast in the radial functional setting. For $N=3$ and $\sigma>2$ or $N=4$ and $\sigma>1$, Lemma \ref{psbis}, considered in the radial setting, provides the existence of a Palais-Smale sequence $(u_n)_n \subset  \mathcal{M}_{rad}(c)$  for $E$ restricted to $ S_{rad}(c)$ which minimizes $E$.  Up to a subsequence, it weakly converges to a $u_c \in H^2_{rad}(\R^N)$ which solves \eqref{eq:min-alphac} for a $\alpha \in \R$. Now by Lemma  \ref{conv}, we know that, up to a subsequence, $(u_n)_n$ converges strongly in $H^2(\R^N)$ provided it converges strongly in $L^{2 \sigma +2}(\R^N)$ and $\alpha >0$. 
Note that the strong convergence in $L^{2 \sigma +2}(\R^N)$ follows from the compact embedding of
$H^2_{rad}(\R^N)$ into $L^{2 \sigma +2}(\R^N)$ and that using this convergence, one can easily deduce that $E(u_c) = \Gamma_{rad}(c)$. To end the proof, we can assume without restriction that $\|u_c\|_2^2:=d <c$ (otherwise we are done). Since $\gamma_{rad}$ is nonincreasing, see Remark \ref{rem:gammaradmonotone}, it then follows that $E(u_c)= \gamma_{rad}(d)$ and Corollary \ref{sign-la2} applies. We deduce that $\alpha >0$ and thus  $(u_n)_n $ strongly converges in $H^2(\R^N)$ to $u_c$. In particular $u_c \in S(c)$. 
\end{proof}

\begin{cor}\label{gammaradmonostrict}
 If $N=3$ and $\sigma\geq 2$ or $N=4$ and $\sigma\ge 1$, then the map $c \mapsto \gamma_{rad} (c)$ is strictly decreasing. 
\end{cor}
\begin{proof}
One just need to argue as in the proof of Corollary \ref{gammamonostrict} taking Corollary \ref{sign-la2} into account.  
\end{proof}

\section{A concentration phenomenon, proof of Theorem \ref{concentration}}\label{concentrationpart}

In this section, we establish the concentration of solutions, as $c \mapsto c_N^*$, described in Theorem \ref{concentration}. As a preliminary result we derive the following lemma.

\begin{lem}\label{ground}
Let $\sigma N=4$ and $u \in H^2(\R^N)$ be a non trivial solution of
\begin{align} \label{nls}
\gamma \Delta^2 u + u =|u|^{\frac 8N}u.
\end{align}
Then $\|u\|_2^2 \geq c^*_N$. Moreover, if $\|u\|_2^2=c^*_N$, then $u$ is a minimizer of the functional
$$
F(u):=\frac {\gamma}{2} \int_{\R^N}|\Delta u|^2\, dx + \frac 12 \int_{\R^N}|u|^2 \, dx - \frac{N}{2N+8} \int_{\R^N}|u|^{2 + \frac 8N} \, dx,
$$
 on its Nehari set
\begin{equation*}
\mathcal{N}:= \{u \in H^2(\R^N) \backslash \{0 \} : F'(u)u =0 \}.
\end{equation*}

\end{lem}
\begin{proof}
We know from Lemma \ref{Pohozaevs} that
\begin{align}\label{p2}
\gamma \int_{\R^N}|\Delta u|^2 \, dx = \frac {N}{N+4} \int_{\R^N}|u|^{2 + \frac 8N}\, dx,
\end{align}
which implies that
\begin{align}\label{energy}
F(u)=\frac 12 \int_{\R^N}|u|^2 \, dx.
\end{align}
Observe now that if $u\ne 0$ solves \eqref{nls}, then $\|u\|_2^2 \geq c^*_N$. Indeed, using the Gagliardo-Nirenberg inequality \eqref{G-N-H2-ineq}, we deduce from  \eqref{p2} that
$$
\gamma \int_{\R^N}|\Delta u|^2 \, dx \leq \left(\frac{\|u\|_2^2}{c^*_N}\right)^{\frac 4N} \gamma \int_{\R^N}|\Delta u|^2 \, dx.
$$
This shows that $\|u\|_2^2 \geq c^*_N$ and the last assertion follows from \eqref{energy}.
\end{proof}

We are now in a position to prove Theorem \ref{concentration}.
\begin{proof}[Proof of Theorem \ref{concentration}]
By Theorem \ref{thmmain}, there exist sequences $(c_n)_n \subset (c_N^\ast ,\infty )$ with $c_n \to c^*_N$  and $(u_n)_n \subset {\mathcal{M}}(c_n)$ such that  $E(u_n) = \Gamma(c_n)$. From (\ref{explosion}), we deduce that
\begin{align} \label{limit}
\int_{\R^N}|\Delta u_n|^2 \, dx \rightarrow \infty \,\, \text{as} \, \, n \rightarrow \infty 
\end{align}
and we therefore infer from Remark \ref{Rem:equivlapgrad} that
\begin{align}\label{limit1}
\frac{\int_{\R^N}|\nabla u_n|^2 \, dx }{\int_{\R^N}|\Delta u_n|^2 \, dx}
\le \frac{ c_n^{1/2}}{\left(\int_{\R^N}|\Delta u_n|^2 \, dx\right)^{1/2}}
 \rightarrow 0 \, \, \text{as} \, \/
n \rightarrow \infty.
\end{align}
Since $Q(u_n)=0$, we know that
\begin{align}\label{limit2}
\frac{\int_{\R^N}| u_n|^{2+ \frac 8N} \, dx }{\gamma \int_{\R^N}|\Delta u_n|^2 \, dx} =\frac {N+4}{N}. 
\end{align}
Next, we introduce $\tilde{u}_n(x):= \eps_n^{\frac N2}u_n(\eps_n x)$ where we chose
\begin{align}\label{def}
\eps_n^{-4}=\gamma \int_{\R^N}|\Delta u_n|^2 \, dx \to \infty \ \text{as} \ n \to \infty.
\end{align}
Direct calculations show that  $\|\tilde{u}_n\|_2^2 =\|u_n\|_2^2=c_n, \|\Delta \tilde{u}_n\|_2^2 = \frac{1}{\gamma}$, and
\begin{align}\label{norm}
\int_{\R^N}|\tilde{u}_n|^{2 + \frac 8N}\, dx =\frac {N+4}{N}.
\end{align}
Using \cite[Lemma I.1]{Li2}, we deduce from \eqref{norm} that there exist $\delta>0$ and a sequence $(y_n)_n \subset \R^N$ such that, for some $R>0$,
\begin{align}\label{neq}
\int_{B_R(y_n)}|\tilde{u}_n|^2 \, dx \geq \delta.
\end{align}
Thus, defining
\begin{align} \label{defvn}
v_n(x):=\tilde{u}_n(x+y_n)=\eps_n^{\frac N2}u_n(\eps_n x+ \eps_n y_n),
\end{align}
it follows from \eqref{neq} that there exists $v\ne 0$ such that $v_n\rightharpoonup v$ in $H^2(\R^N)$. Since $u_n$ solves
$$
\gamma \Delta^2 u_n - \Delta u_n +  \alpha_n u_n = |u_n|^{\frac 8N}u_n,
$$
where the Lagrange multiplier is given by
$$
\alpha_n=\frac{1}{c_n}\left( -\gamma \int_{\R^N}|\Delta u_n|^2 \, dx  - \int_{\R^N} |\nabla u_n|^2 \, dx  +\int_{\R^N}|u_n|^{2 + \frac 8N}\, dx\right),
$$
$v_n$ satisfies
$$
\gamma \Delta^2 v_n - \eps_n ^2\Delta v_n+ \eps_n^4 \alpha_n v_n = |v_n|^{\frac 8N}v_n.
$$
Combining \eqref{limit}, \eqref{limit1} and \eqref{limit2}, we deduce that
\begin{equation}\label{explosionbis}
\eps_n^4 \alpha_n \rightarrow \frac{4}{c^*_NN} \, \, \text{as}\, \, n \rightarrow \infty.
\end{equation}
Since $v_n\rightharpoonup v$ in $H^2(\R^N)$ as $n \rightarrow \infty$, $v$ solves
\begin{align}\label{v}
\gamma \Delta^2 v + \frac{4}{c^*_NN} v=|v|^{\frac 8N}v.
\end{align}
Now setting
$$ w_n(x):=\left( \frac{c^*_N N}{4}\right)^{\frac N8}v_n\left(\left( \frac{c^*_NN}{4}\right)^{\frac 14} x\right) \mbox{ and }
u(x):=\left( \frac{c^*_NN}{4}\right)^{\frac N8}v\left(\left( \frac{c^*_NN}{4}\right)^{\frac 14} x\right),$$
it is easily seen that $w_n \rightharpoonup u$ in $H^2(\R^N)$ as $n \to \infty,$ and $\|w_n\|_2^2=\|v_n\|_2^2 = c_n$.
Moreover it follows from \eqref{v} that $u$ is solution to \eqref{nls} and thus by Lemma \ref{ground},
 $ \|u\|_2^2 \geq c_N^*.$
On the other hand, since $w_n \rightharpoonup u$ in $H^2(\R^N)$ as $n \to \infty$, we have
$
\|u\|_2^2 \leq \liminf_{n \to \infty} \|w_n\|_2^2 =c^*_N
$ and thus we obtain that $\|u\|_2^2=c^*_N$. By Lemma \ref{ground}  $u$ is a least energy solution to \eqref{nls}. Since $\|u\|_2^2=c^*_N, \|w_n\|_2^2=c_n \rightarrow c^*_N$ as $n \rightarrow \infty$, and $w_n \rightharpoonup u$ in $H^2(\R^N)$ as $n \to \infty,$ it follows that $ w_n \to u  \ \text{in} \ L^2(\R^N) \ \text{as} \ n \rightarrow \infty.$
Now from the definition \eqref{defvn}, and by the interpolation inequality in Lebesgue space, there holds for $2 \leq q < 4^\ast$,
$$
\left(\frac{\eps_n^4 c^*_NN}{4}\right)^{\frac N8}u_n\left(\left(\frac{\eps_n^4 c^*_NN}{4}\right)^{\frac 14} x+ \eps_n y_n\right)
\rightarrow u \ \text{in} \ L^q(\R^N) \ \text{as} \ n \rightarrow \infty.
$$
This  completes the proof.
\end{proof}

\section{Positive and sign-changing solutions, proof of Theorems \ref{thm10} and \ref{thm11} }\label{special}
In this section, we prove Theorem \ref{thm10} and Theorem \ref{thm11}. Theorem \ref{thmmain} gives the existence of a ground state for $c\in (c_0,c_{N,\sigma})$.
To show that when $\sigma \in \N$, one of them is radial we make use of the  Fourier rearrangement, introduced in \cite{BoLe}, that we now recall. For $u \in L^2(\R^N)$, let $u^{\sharp}$ be the Fourier rearrangement of $u$ defined by
$$
u^{\sharp}:=\mathcal{F}^{-1}((\mathcal{F}u)^{\ast}),
$$
where $\mathcal{F}$ and $\mathcal{F}^{-1}$ denote respectively the Fourier transform and the inverse Fourier transform, and $f^{\ast}$ stands for the Schwarz rearrangement of $f$. Notice that $u^{\sharp}$ is radial and  $\|u^{\sharp}\|_2=\|u\|_2$. Moreover, we recall that \cite[Lemma A.1]{BoLe} implies
\begin{equation}\label{rearrangement}
\|\Delta u^{\sharp}\|_2 \leq \|\Delta u\|_2, \quad \|\nabla u^{\sharp}\|_2  \leq \|\nabla u\|_2 \quad \mbox{and} \quad  \|u^{\sharp}\|_{2\sigma+2}  \geq \|u\|_{2\sigma + 2}.
\end{equation}
\begin{proof}[Proof of Theorem \ref{thm10}]
Let $c\in (c_0,c_{N,\sigma})$
and let $u_c$ be a ground state associated to $\Gamma(c)$.  From \eqref{rearrangement}, we obtain that $Q(u_c^{\sharp}) \leq Q(u_c)=0$. Hence by Lemma \ref{unique}, there exists a $\lambda \in(0,1]$ such that $Q((u_c^{\sharp})_{\lambda}) =0$. Observe that
\begin{align*}
\Gamma(c) \leq E((u_c^{\sharp})_{\lambda})
&=E((u_c^{\sharp})_{\lambda}) - \frac{2}{\sigma N} Q((u_c^{\sharp})_{\lambda}) \\
&=\lambda^2\dfrac{\sigma N-4}{2\sigma N} \gamma\int_{\R^N}|\Delta u_c^{\sharp}|^2\, dx +\lambda \dfrac{\sigma N -2}{2\sigma N} \int_{\R^N}|\nabla u_c^{\sharp}|^2\, dx\\
&\le\lambda^2\dfrac{\sigma N-4}{2\sigma N} \gamma\int_{\R^N}|\Delta u_c|^2\, dx +\lambda \dfrac{\sigma N -2}{2\sigma N} \int_{\R^N}|\nabla u_c|^2\, dx\\
& \leq E(u_c)- \frac{2}{\sigma N} Q(u_c)=\Gamma(c),
\end{align*}
whence $\lambda=1$ and $E(u_c^{\sharp})=\Gamma(c)$. Therefore, $u_c^{\sharp}$ is also a ground state solution to \eqref{Pc}. It remains to prove that $u_c^{\sharp}$ is sign-changing. Associated to $u_c^{\sharp}$  there exists a Lagrange multiplier $\alpha_c \in \R$ so that
$$
\gamma \Delta^2 {u}_c^{\sharp} - \Delta {u}_c^{\sharp}+ {\alpha}_c {u}_c^{\sharp}=|{u}_c^{\sharp}|^{2\sigma}{u}^{\sharp}_c.
$$
When $4 < \sigma N <4^*$, we deduce from (\ref{nonexe2}) and (\ref{divergence}) that $\alpha_c \to + \infty$ as $c \to 0$ whereas the same conclusion holds when $\sigma N =4$, as shown by  (\ref{explosion}), when  $c_n \to {c^*_N}^+$. In particular, this implies that $2\sqrt{\gamma \alpha_c}>1$ when $c$ is sufficiently small (or $c_n> {c^*_N}^+$ sufficiently close to ${c^*_N}^+$ when $\sigma N=4$). At this point, using \cite[Theorem 3.7]{BoCaMoNa}, we deduce that ${u}_c^{\sharp}$ is sign-changing for $c>0$ small enough.
\end{proof}

\begin{proof}[Proof of Theorem \ref{thm11}]
We borrow here an idea from \cite{BoCaMoNa}. We consider the modified minimization problem 
\begin{align}\label{appmin}
{\Gamma}^+(c):= \inf_{u \in {\mathcal{M}^+}_{rad}(c)} {E}^+(u),
\end{align}
where
\begin{align*}
{E}^+(u):&=\frac{\gamma}{2}\int_{\R^N}|\Delta u|^2\, dx+\frac{1}{2}\int_{\R^N}|\nabla u|^2\, dx-\frac{1}{2\sigma+2}\int_{\R^N}|u^+|^{2\sigma+2}\, dx, \\
{Q}^+(u):&=\gamma  \int_{\R^N}|\Delta u|^2\, dx+\dfrac{1}{2}\int_{\R^N}|\nabla u|^2\, dx-\frac{\sigma N}{2(2\sigma+2)}\int_{\R^N}|u^+|^{2\sigma+2}\, dx,
\end{align*}
and 
$$ \mathcal{M}^+_{rad}(c) : = \{u \in S(c) : {Q}^+(u) =0\}\cap H^2_{rad}(\R^N).$$
It is straightforward to check that the equivalent of Theorem \ref{radN34} holds for the problem
\begin{equation}\label{Pc+}\tag{$P_c^+$} \gamma \Delta^2 u - \Delta u +  \alpha u = |u^+|^{2 \sigma}u^+  \quad \mbox{ with } \quad  \int_{\R^N}|u|^2 dx = c.
\end{equation}
Namely, if $N\le 4$ with $4\le\sigma N$, there exists a radial ground state solution ${u}_c \in H^2(\R^N)\backslash \{0\}$ to \eqref{Pc+} for any $c>c_0$. Observe also that the associated Lagrange multiplier $\alpha^+_c$ is positive for every $c>c_0$. 

We claim that $\alpha^+_c\to 0$ as $c\to \infty$. Indeed, we compute 
\begin{equation}\label{al}
\alpha^+_c =\frac{1}{c} \left(- 2{E}^+(\bar{u}_c) + \frac{\sigma}{\sigma +1} \int_{\R^N}|{u}^+_c|^{2\sigma+2}\, dx \right)
 \leq  \frac{\sigma}{c(\sigma +1)} \int_{\R^N}|{u}^+_c|^{2\sigma+2}\, dx
\end{equation}
so that the claim will be proven if we show that 
\begin{equation}\label{limsupalphac}
\limsup_{c\to\infty}  \int_{\R^N}|{u}^+_c|^{2\sigma+2}\, dx<\infty. 
\end{equation}
Since $\Gamma(c)\ge 0$, we infer that
\begin{equation}\label{ajout4}
\int_{\R^N}|{u}^+_c|^{2\sigma+2}\, dx \leq (\sigma +1) \left( \gamma \int_{\R^N}|\Delta u^+_c|^2\, dx+\int_{\R^N}|\nabla u^+_c|^2\, dx  \right).
\end{equation}%
On the other hand, adapting the proof of Lemma \ref{lemgammadecr}, we can show that $\gamma^+(c)$ is a nonincreasing function. Thus, we deduce from (\ref{relEQ1}),  that 
\begin{equation}\label{ajout5}
\gamma\dfrac{\sigma N-4}{2\sigma N}\int_{\R^N}|\Delta u^+_c|^2\, dx +\dfrac{\sigma N -2}{2\sigma N}\int_{\R^N}|\nabla u^+_c|^2\, dx, \quad \mbox{remain bounded.} 
\end{equation}
If $\sigma N >4$, the combination of \eqref{ajout4} and \eqref{ajout5} shows immediately that \eqref{limsupalphac} holds. If $\sigma N =4$, we first deduce from \eqref{ajout5} that the quantity $\int_{\R^N}|\nabla u^+_c|^2\, dx$ remains bounded and then Lemma \ref{boundcritical} permits to conclude that it is also the case for $\int_{\R^N}|\Delta u^+_c|^2\, dx$. Thus \eqref{limsupalphac} also holds when $\sigma N =4$.

Now using the fact that $\alpha^+_c > 0$ is small when $c >c_0$ is sufficiently large, we may rewrite the equation in \eqref{Pc+} as a system
\begin{align*}
     \begin{cases}
     & - \gamma \Delta {u}_c + \lambda_1 {u}_c ={v}_c,\\
     & - \Delta {v}_c +  \displaystyle \frac{\lambda_2}{\gamma} {v}_c=|\bar{u}_c^+|^{2\sigma}\bar{u}_c^+,
     \end{cases}
\end{align*}
where $\lambda_1, \lambda_2 \geq 0$ satisfy $\lambda_1 \lambda_2= \gamma \alpha^+_c$, and $\lambda_1 + \lambda_2=1$.
It is then standard, by the strong maximum principle, to deduce that ${u}_c >0$ and in particular ${u}_c $ satisfies (\ref{4nls}). 
\end{proof}

\begin{remark}
Observe that, when $N=1,2$ and $ 4 \leq \sigma N < \infty$ or $4/3\leq \sigma \leq 2$ and $N=3$ or $N=4$ and $\sigma =1$, one can work directly in $${\mathcal{M}^+}(c) : = \{u \in S(c) : {Q}^+(u) =0\},$$
and prove that any solution associated to ${\Gamma}^+(c):= \inf_{u \in {\mathcal{M}^+}(c)} {E}^+(u)$ is radially symmetric and positive. Indeed, under our assumptions on $(N, \sigma)$, the conclusions of Lemma \ref{sign-la} are available and thus the analogue of Theorem \ref{thmmain}  yields the existence of a solution associated to $\gamma^+(c)$  for any $c>c_0$. Proceeding as in the proof of Theorem \ref{thm11} one can show that this solution is positive. Now to show that the solution is radially symmetric around some point, we set 
$$
f(u, v):= (\frac {1}{4\gamma}- \alpha^+_c ) u - \frac {1}{2 \gamma} v + |u|^{2 \sigma} u, \quad
g(u, v):= v- \frac 12 u.
$$
Then, (\ref{4nls}) is equivalent to the cooperative  system
\begin{align*}
     \begin{cases}
     & \gamma \Delta u + g(u, v) =0,\\
     & \Delta v +  f(u, v)=0.
     \end{cases}
\end{align*}
Since we are in the setting of Busca-Sirakov \cite{BuSi}, we can apply \cite[Theorem 2]{BuSi}, to deduce that any positive solution is radially symmetric around some point.
\end{remark}

\section{The Dispersive equation, proof of Theorems \ref{globex} and \ref{unstable}}\label{dispersiveeq}
This section is devoted to the dynamics of the solutions $u=u(t,x)$ to the dispersive equation \eqref{4nlsdis}.
First, we exhibit a class of initial data for which solutions to \eqref{4nlsdis} exists globally in time. We then discuss the instability of the standing waves associated to radial ground state solutions. \medskip 

We start by recalling the local well-posedness of the solutions to \eqref{4nlsdis} and a blow-up alternative due to \cite{Pa},
\begin{lem}(\cite[Proposition 4.1]{Pa}) \label{alternative}
Let $ 0 <\sigma N < 4^*$. For any $u_0\in H^2 (\R^N)$, there exist $T>0$ and a unique solution  $u \in C([0,T); H^2 (\R^N))$ to \eqref{4nlsdis} with initial datum $u_0$ so that the mass and the energy are conserved along time, that is for any $t \in [0, T)$
$$
\|u(t)\|_{2}=\|u_0\|_{2}, \mbox{ and } \ E(u(t))=E(u_0).
$$
Moreover,  either $T=\infty,$ or $\lim_{t\rightarrow T^-}\|\Delta u(t) \|_2=\infty.$
\end{lem}

\begin{proof} [Proof of Theorem \ref{globex}]
Let $c >c_0$ be arbitrary. First, observe that $\mathcal{O}_c \neq \emptyset$. 
Indeed, we know from Lemma \ref{Mnonvoid} that if $c>c_0$ then $\mathcal{M}(c) \neq \emptyset.$ Fixing $v \in \mathcal{M}(c)$, we have by definition that $E(v) \geq \Gamma(c) >0$ and $Q(v)=0$. From the scaling \eqref{functional}, we see that $E(v_{\lambda}) \to 0$ as $\lambda \to 0$. Note also that
$$ 
Q(v_{\lambda})=( \lambda^2 - \lambda^{\frac{\sigma N}{2}}) \gamma \int_{\R^N}|\Delta u|^2\, dx
 + \frac{\lambda - \lambda^{\frac{\sigma N}{2}}}{2}\int_{\R^N}|\nabla u|^2\, dx,$$
which proves that, for any $ 4 \leq \sigma N < 4^*$, $Q(v_{\lambda}) \to 0^+$ as $\lambda \to 0$. Thus, taking $\lambda >0$ small enough, we have that $v_{\lambda} \in \mathcal{O}_c$.

Let $u_0\in \mathcal{O}_c$  and denote by $u\in C([0,T);H^2 (\R^N))$ the solution to \eqref{4nlsdis} with initial datum $u_0$. We now prove that $u$ exists globally in time, i.e. $T= \infty$. Suppose by contradiction that $T < \infty$. From Lemma \ref{alternative}, we infer that
\begin{align} \label{blowup1}
\displaystyle\lim_{t\rightarrow T^-}\int_{\R^N}|\Delta u(t)|^2 \, dx =\infty.
\end{align}
Observe that $E(u(t))=E(u_0)$ for $0 \leq t<T$, and
\begin{equation}\label{conserve}
E(u(t))-\frac{2}{\sigma N} Q(u(t))=\gamma\dfrac{\sigma N-4}{2\sigma N}\int_{\R^N}|\Delta u|^2\, dx +\dfrac{\sigma N -2}{2\sigma N}\int_{\R^N}|\nabla u|^2\, dx.
\end{equation}
Thus, when $4 < \sigma N < 4^*$, we deduce from (\ref{blowup1}) that 
\begin{align} \label{blowup}
\lim_{t\rightarrow T^-} Q(u(t))=-\infty.
\end{align}
When $\sigma N =4$, using the fact that both the energy and the mass are conserved,  Lemma \ref{boundcritical} applies and shows
\begin{align*}
\displaystyle\lim_{t\rightarrow T^-}\int_{\R^N}|\nabla u(t)|^2 \, dx =\infty.
\end{align*}
Again we deduce from (\ref{conserve}) that (\ref{blowup}) holds.

By  continuity, we infer that there exists $t_0\in (0,T)$ such that $Q(u(t_0))=0$. Since $\|u(t_0)\|_2=\|u_0\|_2 = c$,  we have $E(u(t_0)) \geq \Gamma(c)$ by definition $\gamma$. This contradicts the fact that  $E(u(t_0))=E(u_0)<\Gamma(c)$. 
\end{proof}

Let us now prove  Theorem \ref{unstable}. To this aim, we first recall the localized virial identity introduced in \cite{BoLe}, namely
$$
M_{\varphi_R} [u]:= 2 \text{Im} \int_{\R^N}\bar{u} \nabla \varphi_R \nabla u\ dx,
$$
where $u \in H^2 (\R^N)$, $\varphi : \R^N \rightarrow \R$ is a radial function such that $\nabla^j \varphi \in L^\infty (\R^N)$ for $1\leq j \leq 6$,
$$
\varphi (r):=
\begin{cases}\dfrac{r^2}{2} & for\ r\leq 1 \\
 const. & for\ r\geq 10
\end{cases},
\ \varphi'' (r)\leq 1 \ for\ r\geq 0,
$$
and $\varphi_R (r):= R^2 \varphi (\dfrac{r}{R})$ for $R>0$. 

In \cite[Lemma 3.1]{BoLe}, it is proved, for $N \geq 2$, that if $u \in C([0,T); H^2 (\R^N))$ is the radial solution to \eqref{4nlsdis} with initial datum $u_0 \in H^2_{rad}(\R^N)$, then
\begin{align} \label{locvirial}
\begin{split}
\dfrac{d}{dt}M_{\varphi_R} [ u(t) ] &\leq 4 N\sigma E(u_0) - (2N\sigma -8)\gamma \|\Delta u(t)\|_2^2 - (2N\sigma -4 ) \|\nabla u(t)\|_2^2\\
&  +O\left(\dfrac{\|\nabla u(t)\|_2^2}{R^2}+ \dfrac{\|\nabla u(t)\|_2^\sigma}{R^{\sigma (N-1)}}+\dfrac{1}{R^2}  + \dfrac{1}{R^4}\right)\\
&= 8 Q(u(t)) +O\left(\dfrac{\|\nabla u(t)\|_2^2}{R^2}+ \dfrac{\|\nabla u(t)\|_2^\sigma}{R^{\sigma (N-1)}}+\dfrac{1}{R^2}+ \dfrac{1}{R^4} \right).
\end{split}
\end{align}

\begin{proof} [Proof of Theorem \ref{unstable}] Suppose that $u_c$ is a radial ground state solution, and define
$$
\Theta: =\{v\in H_{rad}^2(\R^N) \backslash \{0\}: E(v)<E(u_c),\ \|v \|_2 = \|u_c \|_2,\ Q(v)<0  \}.
$$
The set $\Theta$ contains elements arbitrarily close to $u_c$ in $H^2(\R^N)$. Indeed letting  $v_0:= (u_c)_\lambda$ we see from Lemma \ref{unique} that $v_0 \in \Theta$ if $\lambda >1$ and that  $v_0 \to u_c$ in $H_{rad}^2(\R^N)$ as $\lambda \to 1^+$. Let $v \in C([0, T); H_{rad}^2(\R^N))$ be the solution to \eqref{4nlsdis} with radial initial datum $v_0$ and $T\in (0,\infty]$ be the maximal existence time. To prove the theorem, we just need to show that $v(t)$ blows up in finite time. We divide the remaining arguments of the proof into four steps.\medskip \\
{\bf First Step --} We claim that there exists $\beta >0$ such that $Q(v(t)) \leq - \beta$ for any $t \in [0, T)$. Indeed, arguing as in the proof of Theorem \ref{globex}, we easily check that  $v(t)\in \Theta$ and in particular $Q(v(t))<0$ for any $t\in [0,T)$. Now setting $v:=v(t)$, in view of Lemma \ref{unique}, since $Q(v)<0$, there exists $\lambda^\ast<1$ such that $Q(v_{\lambda^\ast})=0$. Moreover the function $\lambda \mapsto E(v_{\lambda})$ is concave for $ \lambda \in [\lambda^\ast ,1]$, thus
$$
E(v_{\lambda^\ast})-E(v) \leq (\lambda^\ast -1) \dfrac{\partial E (u_\lambda)}{\partial \lambda}|_{\lambda=1}
=(\lambda^\ast -1) Q(v).
$$
Since $Q(v)<0$, $E(v)=E(v_0)$ and $v_{\lambda^\ast}\in \mathcal{M}_{rad}(c)$, we deduce that
\begin{align} \label{qvt}
Q(v)\leq (1-\lambda^\ast)Q(v) \leq E(v)-E(v_{\lambda^\ast})\leq E(v_0)-E(u_c)=:-\beta.
\end{align}
{\bf Second Step --} Suppose that $\sigma \leq 4$. We claim that there exists $\delta>0$ such that
\begin{equation}\label{thmsecondstep}
\dfrac{d}{dt}M_{\varphi_R} [v (t)]\leq - \delta \|\nabla v (t)\|_2^2 \ \text{for} \ t \in [0,T),
\end{equation}
and $t_1 \geq 0$ such that
\begin{equation}\label{initial}
M_{\varphi_R} [v (t)]<0  \ \text{for} \ t \geq t_1.
\end{equation}
To prove
\eqref{thmsecondstep}, we need to distinguish two cases.  \medskip \\
{\it Case 1:} Let
$$
T_1:=\{t \in [0, \infty): (\sigma N-2)\| \nabla v(t)\|_2^2 \leq  4 N \sigma E(v_0)\}.
$$
It follows from \eqref{locvirial} and the First Step that 
\begin{equation}\label{esti1}
\dfrac{d}{dt}M_{\varphi_R} [ v (t) ] \leq - 7 \beta \leq  - \delta ||\nabla v (t)||_2^2
\end{equation}
for some $\delta >0$ sufficiently small and $R >0$ sufficiently large. \medskip \\
{\it Case 2:} Set
$$
T_2:=[0, \infty)\backslash T_1=\{t \in [0, \infty): (\sigma N-2)\|\nabla v(t)\|_2^2 >  4N \sigma E(v_0)\}.
$$
Using the interpolation inequality (\ref{interpolation}), we infer from \eqref{locvirial} that
\begin{align*} \label{locvirial2}
\begin{split}
\dfrac{d}{dt}M_{\varphi_R} [ v (t) ] &\leq  - (N\sigma -2) \|\nabla v (t)\|_2^2 - \frac{(2N\sigma -8)\gamma}{||v_0||_2^2} \|\nabla v (t)\|_2^4   \\
& + O\left(\frac {1}{R^{4}}+ \dfrac{\|\nabla v (t)\|_2^2}{R^2}+ \dfrac{\|\nabla v (t)\|_2^\sigma}{R^{\sigma (N-1)}}+\dfrac{\mu}{R^2} \right) .
\end{split}
\end{align*}
Taking $R$ large enough and noticing that under our assumptions, we have $\sigma \leq 2$ if $N \sigma =4$ and $\sigma \leq 4$ if $\sigma N >4$, we deduce that
\begin{equation} \label{esti2}
\dfrac{d}{dt}M_{\varphi_R} [ v (t) ]  \leq -\dfrac{ (N \sigma  -2)}{2}   \|\nabla v (t)\|_2^2.
\end{equation}
Combining \eqref{esti1} and \eqref{esti2}, we conclude that there exists $\delta>0$ such that (\ref{thmsecondstep}) holds. Finally, since
$$ M_{\varphi_R} [ v (t) ] = M_{\varphi_R} [ v_0  ] + \int_{0}^{t_1} \dfrac{d}{ds}M_{\varphi_R} [ v (s) ] ds,$$
the inequality \eqref{initial} follows
from the estimate
$$ \big| \dfrac{d}{dt}M_{\varphi_R} [ v (t) ] \big| \geq \min \left\{ 7 \beta, \frac{(N \sigma -2)}{2}||\nabla v(t)||_2^2 \right\}.$$
\medskip \\
{\bf Third Step --} We now conclude that the solution $v(t)$ to \eqref{4nlsdis} with initial datum $v_0$ blows up in finite time when $\sigma \leq 4$. For that purpose, we adapt another argument from \cite{BoLe}. Suppose by contradiction that $T=\infty$. Then, integrating \eqref{thmsecondstep} on $[t_1 ,t]$, and taking  \eqref{initial} into account, we have that
$$
M_{\varphi_R} [ v (t) ] \leq  -\delta \int_{t_1}^t \|\nabla v (s)\|_2^2 ds.\
$$
Now using the Cauchy-Schwarz's inequality, we get from the definition of $M_{\varphi_R} [ v (t) ] $  that
$$
|M_{\varphi_R}[v (t)] | \leq 2 \|\nabla \varphi_R\|_{\infty} \|v(t)\|_2 \|\nabla v(t)\|_2 \leq C \|\nabla v (t)\|_2.
$$
Thus, for some $\tau >0$,
\begin{equation}\label{control}
M_{\varphi_R}[v (t)] \leq - \tau \int_{t_1}^t | M_{\varphi_R}[v (s)] |^2 ds.
\end{equation}
Setting $z(t):=\int_{t_1}^t |M_{\varphi_R}[v(s)] |^2 ds $, we obtain from \eqref{control} that
$
z'(t)\geq \tau^2 z(t)^2.
$
Integrating this equation, we deduce that $M_{\varphi_R}[v (t)]\rightarrow -\infty$ when $t$ tends to some finite time $t^\ast$. Therefore, the solution $v(t)$ cannot exist for all $t>0$. By the blow-up alternative recalled in Lemma \ref{alternative} this ends the proof of the theorem when $\sigma \leq 4$.\\

{\bf Fourth Step --} The case $\sigma>4$. First, observe that if $\|\nabla v (t)\|_2$ is unbounded for $t\in [0,T)$, then $\|\Delta v (t)\|_2$ is unbounded for $t\in [0,T)$. Therefore $v (t)$ blows up either in finite or infinite time. On the other hand, if $\|\nabla v (t)\|_2$ is bounded for $t\in [0,T)$, since $Q(v (t))<-\beta$ by the First Step, we can choose $R\geq 1$ sufficiently large to deduce from \eqref{locvirial} that
\begin{equation}
\label{lasteq}
\dfrac{d}{dt}M_{\varphi_R} [v (t)]\leq -4 \beta.
\end{equation}
As in \cite{BoLe}, suppose by contradiction that $T=\infty$. Then, there exists $t_1\geq 0$ such that 
$M_{\varphi_R} [v (t)] <0$ for all $t\geq t_1$. Using Cauchy-Schwarz's inequality and integrating \eqref{lasteq} between $[t_1 ,t]$, we find
$$-\|\nabla \varphi_R \|_{\infty} \|v (t) \|_2^{3/2} \|\Delta v (t)\|_2^{1/2}\leq M_{\varphi_R}[v (t)]\leq - 4a (t-t_1).$$
Therefore, we see that either $v (t)$ blows up in finite time or that
$$\|\Delta v (t)\|_2^2 \geq C (t-t_1)^2,\ \mbox{for all } t\geq t_1 .$$
This completes the proof.
\end{proof}

\section{Appendix}\label{Appendix}
The following lemma is proved in \cite{BoCaGoJe}.
\begin{lem}\label{Pohozaevs}
Let $0 < \sigma N < 4^*$. If $v \in H^2(\R^N)$ is a weak solution of 
\begin{equation}\label{equationgeneral}
\gamma \Delta^2 v - \mu \Delta v +  \omega v = d |v|^{2 \sigma}v
\end{equation}
where $\gamma, \mu, \omega, d$ are constants, then $v$ satisfies $I(v) = P(v)= Q(v)= 0$ where
$$I(u)=\gamma\|\Delta u\|_2^2+\mu\|\nabla u\|_2^2+\omega \|u\|_2^2  - d\|u\|_{2\sigma +2}^{2\sigma+2}. $$
$$
P(u)=\frac{(N-4)\gamma}{2}\|\Delta u\|_2^2+ \frac{(N-2)\mu}{2}\|\nabla u\|_2^2+ \frac{N \omega}{2} \|u\|_2^2  - \frac{dN}{2 \sigma +2}\|u\|_{2\sigma +2}^{2\sigma+2},
$$
and
$$Q(u)=\gamma\|\Delta u\|_2^2+\frac{\mu}{2}\|\nabla u\|_2^2-\frac{d\sigma N}{2(2\sigma+2)}\|u\|_{2\sigma +2}^{2\sigma+2}.$$
\end{lem}
\begin{remark}\label{rem:pohoX}
Observe that the previous lemma also holds if we assume that $v\in X$ ($X$ being defined in \eqref{defX}) and $\omega=0$.
\end{remark}

\begin{prop} \label{finitemass}
Let $N\geq 5$ and $ \sigma \in (\frac{2}{N-2}, \frac{4}{N-4})$. Then any solution $u\in X$ to \eqref{4NLSalpha=0} belongs to $L^2(\R^N)$.
\end{prop}

\begin{proof}
We can assume without loss of generality that $\gamma =1$.
The main idea of the proof consists in testing \eqref{4NLSalpha=0} with a function $\varphi^2 u$ where,  roughly,  $\varphi(x)=1+|x|$.\smallskip

Let $\psi\in C^\infty (\R^N)$ with $supp \ \psi \subset \R^N \backslash B_R(0)$ be such that $\psi(x) = 1$ for $|x| \geq 2R$. Here $R>0$ is a constant to be determined later. For $R_1>2R$, we define $\varphi:=\psi h_{R_1}$, where $h_{R_1}\in C^2 (\R^N)$ satisfies 
\[ h_{R_1} (x) =
\left\{
\begin{aligned}
&|x|  &\text{ for }2R \leq |x|< R_1 ,\\
& R_1 \left(1+  th \left( \frac{|x|-R_1}{R_1} \right)\right)  &\text{ for }|x|\geq R_1.
\end{aligned}
\right.
\]
Note that in the definition of $h_{R_1}$, $th$ denotes the hyperbolic tangent. Now let
\begin{align} \label{lamax}
\lambda_1(R_1):=\sup_{|x| \geq 2R}\dfrac{|x||\nabla \varphi (x)|}{\varphi (x)},
\quad \lambda_2(R_1):=\sup_{|x| \geq 2R}\dfrac{|x||\Delta \varphi (x)|}{\varphi (x)}.
\end{align}
From the definition of $\varphi$ it readily follows that
$\lambda_1(R_1)=1$ for all $R_1>0$ and that $\lambda_2:=\lambda_2(R_1) \to 0$ as $ R_1 \to \infty.$

As a preliminary step we derive some pointwise  identities. Simple calculations yield
\begin{align*}
\Delta (\varphi^2 u) &=\varphi^2 \Delta u + 4\varphi \nabla u \nabla \varphi +u (2\varphi \Delta \varphi
+2 |\nabla\varphi|^2),
\end{align*}
and
\begin{align*}
(\Delta (\varphi u))^2 &= \varphi^2 (\Delta u)^2 +4 |\nabla \varphi \nabla u|^2+ u^2 (\Delta \varphi)^2 +
4  \varphi \Delta u \nabla \varphi \nabla u \\
& + 2 \varphi u \Delta u \Delta \varphi + 4 u\Delta \varphi \nabla \varphi \nabla u .
\end{align*}
Combining the two previous identities, we obtain
\begin{align}
\label{fdecae1}
(\Delta (\varphi u))^2 &= \Delta u \Delta (\varphi u^2)\nonumber\\
&+ 4 |\nabla \varphi \nabla u|^2+ u^2 (\Delta \varphi)^2   + 4 \nabla \varphi \nabla u u\Delta \varphi-2 u\Delta u |\nabla\varphi|^2.
\end{align}
We shall also use the relation
\begin{equation}
\label{fdecae2}
|\nabla (\varphi u )|^2= \nabla u \nabla (\varphi^2 u)+ |\nabla \varphi|^2 u^2.
\end{equation}
Now testing \eqref{4NLSalpha=0} with $\varphi^2 u$ and using \eqref{fdecae1}-\eqref{fdecae2}, we infer that
\begin{align} \label{testing}
\begin{split}
\int_{\R^N}  |\Delta (\varphi u)|^2+ |\nabla (\varphi u)|^2 \, dx & = \int_{\R^N} |\varphi u|^2 |u|^{2\sigma} + \int_{\R^N} |\nabla \varphi|^2 |u|^2 \, dx  \\
&+ 4\int_{\R^N} |\nabla \varphi \nabla u|^2 \, dx + \int_{\R^N} |u \Delta \varphi|^2 \, dx  \\
&+ 4 \int_{\R^N} u \Delta \varphi \nabla \varphi \nabla u \, dx- 2 \int_{\R^N} u\Delta u |\nabla\varphi|^2 \, dx.
\end{split}
\end{align}
Recalling H\"older inequality and taking into account (\ref{G-N-H1-ineq2}), we obtain
\begin{align*}
\int_{\R^N} |\varphi u|^2 |u|^{2\sigma}\, dx &\leq \left(\int_{|x| \geq R} |u|^{2\sigma +2}\, dx \right)^{\frac{\sigma}{\sigma+1}}
\left(\int_{\R^N} |\varphi u|^{2\sigma +2}\, dx  \right)^{\frac{1}{\sigma +1}}\\
&\leq  C \left(\int_{|x|\geq R} |u|^{2\sigma +2} \right)^{\frac{\sigma}{\sigma+1}}
\int_{\R^N}  |\Delta (\varphi u)|^2+ |\nabla (\varphi u)|^2 \, dx.
\end{align*}
Set $ \delta(R) := C \left(\int_{|x| \geq R} |u|^{2\sigma +2} \right)^{\frac{\sigma}{\sigma+1}}$ and observe that $\delta(R) \to 0$ as $R \rightarrow \infty$. It then follows from \eqref{testing} that
\begin{align}\label{fdecae3}
\begin{split}
(1- \delta(R)) \int_{\R^N}  |\Delta (\varphi u)|^2 + |\nabla (\varphi u)|^2 \, dx
&\leq \int_{\R^N}  |\nabla \varphi|^2 |u|^2 \, dx + 4\int_{\R^N}|\nabla \varphi \nabla u|^2 \, dx \\
&+ \int_{\R^N} |u \Delta \varphi|^2 \, dx + 4\int_{\R^N} u\Delta \varphi \nabla \varphi \nabla u  \, dx\\
&-2 \int_{\R^N} u\Delta u |\nabla\varphi|^2\, dx =: \sum_{i=1}^5 I_i.
\end{split}
\end{align}
We now estimate $I_i$ for $1 \leq i \leq 5$. It comes from  \eqref{lamax} that
 \begin{align*}
I_1=  \int_{\R^N}  |\nabla \varphi|^2 |u|^2 \, dx
&= \int_{|x| < 2R } |\nabla \varphi|^2 |u|^2 \, dx +  \int_{|x| \geq 2R}  |\nabla \varphi|^2 |u|^2 \, dx\\
&  \leq C\int_{|x| < 2R} |u|^2 \, dx +  \int_{|x| \geq 2R} \dfrac {|\varphi  u|^2}{|x|^2}  \, dx.
\end{align*}
Noting that $ \nabla \varphi \nabla(\varphi u)=|\nabla \varphi|^2 u +(\nabla \varphi \nabla u)\varphi,$ it follows for $|x|\geq 2R$, that
$$
|\nabla \varphi \nabla u | \leq \frac{|\nabla \varphi \nabla(\varphi u) |}{\varphi}
+ \frac{|\nabla \varphi|^2}{\varphi^2}|\varphi u|.
$$
Combining this inequality and Young's inequality, we obtain, for any $\eps >0$,
\begin{align*}
&\frac{I_2}{4} =\int_{\R^N}  |\nabla \varphi \nabla u|^2 \, dx \leq \int_{ |x| < 2R } |\nabla \varphi \nabla u|^2 \, dx
+\int_{|x| \geq 2R} |\nabla \varphi \nabla u|^2 \, dx\\
& \leq C \int_{ |x| < 2R } |\nabla u|^2 \, dx +  \int_{|x|\geq 2R }(1+ \eps)\frac{|\nabla \varphi \nabla(\varphi u) |^2}{|\varphi|^2} +  \left(1+\frac {1}{\eps}\right) \frac{|\nabla \varphi|^4 }{|\varphi|^4} |\varphi u|^2\, dx\\
& \leq C \int_{ |x| < 2R } |\nabla u|^2 \, dx + (1+ \eps) \int_{|x|\geq 2R } \dfrac{|\nabla (\varphi u)|^2}{|x|^2} + \left(1+\frac {1}{\eps}\right)\int_{\R^N} \dfrac{|\varphi u|^2 }{|x|^4} \, dx.
\end{align*}
Using \eqref{lamax} again, we get
\begin{align*}
I_3 = \int_{\R^N}|u \Delta \varphi|^2 \, dx
&\leq \int_{|x| < 2R} |u|^2 |\Delta \varphi|^2 \, dx + \int_{|x| \geq 2R} |u|^2 |\Delta \varphi|^2 \, dx \\
& \leq C \int_{|x| < 2R} |u|^2 \, dx + \lambda_2^2 \int_{ |x| \geq 2R }\dfrac{|\varphi u|^2}{|x|^2}  \, dx .
\end{align*}
Next we estimate $I_4$. Given $\eps >0$, Young's inequality leads to
\begin{align*}
I_4 &= 4 \int_{\R^N} u \Delta \varphi \nabla \varphi \nabla u \, dx  \leq 4 \int_{|x| \geq 2R} |u\Delta \varphi||\nabla \varphi \nabla u| \, dx \\
&\leq 2 \eps \int_{|x| \geq 2R} |\nabla \varphi \nabla u|^2 \, dx +\frac{2}{\eps}\int_{|x| \geq 2R}  |u\Delta \varphi|^2 \, dx=\frac{\eps I_2}{2}+ \frac {2I_3}{\eps} \\
& \leq C_{\eps} \int_{|x| <2R} |\nabla u|^2 + |u|^2 \, dx + 2( \eps+ \eps^2) \int_{|x|\geq 2R } \dfrac{|\nabla (\varphi u)|^2}{|x|^2} \, dx \\
&+ 2\left(1 +{\eps} \right) \int_{|x| \geq 2R} \dfrac{|\varphi u|^2 }{|x|^4} \, dx
+ \frac{2 \lambda_2^2}{\eps} \int_{ |x| \geq 2R }\dfrac{|\varphi u|^2}{|x|^2}  \, dx.
\end{align*}
We finally consider $I_5$. For $|x| \geq 2R$, we have
$$
\Delta u=\dfrac{\Delta (u\varphi)}{\varphi} - 2\dfrac{\nabla u \nabla \varphi}{\varphi} + \dfrac{u\Delta \varphi}{\varphi}.
$$
This implies
\begin{align*}
\dfrac{I_5}{2} &\leq \int_{|x| < 2R} |u\Delta u| |\nabla\varphi|^2 \, dx + \int_{|x| \geq 2R} |u\Delta u| |\nabla\varphi|^2  \, dx\\
&\leq C \int_{|x| < 2R } |u\Delta u|  \, dx + \int_{|x| \geq 2R} |\Delta (u\varphi)| \dfrac{|u||\nabla \varphi|^2}{|\varphi|}  \, dx \\
& +2 \int_{|x| \geq 2R} |\nabla u \nabla \varphi| \dfrac{|u||\nabla\varphi|^2}{|\varphi|}  \, dx + \int_{|x| \geq 2R} |u|^2 \dfrac{|\nabla\varphi|^2 |\Delta \varphi|}{|\varphi|}   \, dx\\
&:=C \int_{|x| < 2R} |u\Delta u|  \, dx + \sum_{i=1}^3 J_i .
\end{align*}
By Young's inequality, we have
\begin{align*}		
J_1 &\leq \dfrac{\tau}{2} \int_{|x| \geq 2R} |\Delta (u\varphi)|^2 \, dx+\dfrac{1}{2 \tau } \int_{|x|\geq 2R} |u|^2 \dfrac{|\nabla \varphi|^4}{|\varphi|^2}\, dx\\
&\leq \dfrac{\tau}{2} \int_{|x| \geq 2R} |\Delta (u\varphi)|^2  \, dx +\dfrac{1}{2 \tau}  \int_{|x| \geq 2R} \dfrac{|\varphi u| ^2}{|x|^4} \, dx,
\end{align*}
where $\tau >0$. We also get
\begin{align*}
J_2 & \leq \int_{|x| \geq 2R}|\nabla u\nabla \varphi|^2 \, dx +  \int_{|x| \geq 2R} \dfrac{|u|^2 |\nabla \varphi|^4}{|\varphi|^2} \, dx  \leq \frac{I_2}{4} + \int_{|x| \geq 2R} \dfrac{|\varphi u|^2 }{|x|^4}\, dx \\
& \leq C \int_{ |x| < 2R } |\nabla u|^2 \, dx + \left(2 + \frac{1}{\eps} \right) \int_{|x| \geq 2R} \dfrac{|\varphi u|^2 }{|x|^4} \, dx \\
&+ (1+ \eps) \int_{|x|\geq 2R } \dfrac{|\nabla (\varphi u)|^2}{|x|^2} \, dx\\
\end{align*}
and
$$
J_3=\int_{|x| \geq 2R} |u|^2 \dfrac{|\Delta \varphi| |\nabla\varphi|^2}{|\varphi|} \, dx \leq  \lambda_2 \int_{|x| \geq 2R} \dfrac{| \varphi u|^2}{|x|^3} \, dx.
$$
Pulgging these estimates for the $J_i$'s, we obtain
\begin{align*}
I_5 & \leq C \int_{|x| < 2R}|u|^2 +|\nabla u|^2 \, dx + (2 + 2 \eps) \int_{|x| \geq 2R} \dfrac{|\nabla(\varphi u)|^2}{|x|^2} \, dx \\
& + 2 \lambda_2 \int_{|x| \geq 2R} \dfrac{|\varphi u|^2}{|x|^3} \, dx + \left( 4 + \frac{1}{\tau} + \frac{2}{\eps}\right) \int_{|x| \geq 2R} \dfrac{|\varphi u|^2}{|x|^4} \, dx \\
&+ {\tau} \int_{|x| \geq  2R} |\Delta (u\varphi)|^2 \, dx.
\end{align*}
At last, combining the above estimates for the $I_i$'s, we deduce that
\begin{align*}
&\sum_{i=1}^5 I_i \leq  C(R) + \left(6+ 8 \eps + 2\eps^2 \right)\int_{|x| \geq 2R} \dfrac{|\nabla(\varphi u)|^2}{|x|^2} \, dx + 2 \lambda_2 \int_{|x| \geq 2R} \dfrac{|\varphi u|^2}{|x|^3} \, dx\\
& +  \left(10 + \frac{1}{\tau}+ \frac {6}{\eps} + 2 \eps \right)  \int_{|x| \geq 2R} \dfrac{|\varphi u|^2}{|x|^4} \, dx + \left( 1 + \lambda^2_2 + \frac{2 \lambda_2^2}{\eps}\right) \int_{|x| \geq 2R} \dfrac{|\varphi u|^2}{|x|^2} \, dx \\
&+ {\tau} \int_{|x| \geq 2R} |\Delta (u\varphi)|^2 \, dx.
\end{align*}
Recalling Hardy's inequalities
\begin{align} \label{hardy}
\begin{split}
&\int_{\R^N} |\nabla v|^2 \, dx \geq \left(\dfrac{N-2}{2}\right)^2 \int_{\R^N}\dfrac{|v|^2}{|x|^2} \, dx,\\
&\int_{\R^N} |\Delta v|^2 \, dx \geq \dfrac{N^2}{4} \int_{\R^N}\dfrac{|\nabla v|^2}{|x|^2}\, dx,
\end{split}
\end{align}
we deduce from \eqref{fdecae3} and \eqref{hardy} that
\begin{align}\label{finalcontrol}
\begin{split}
&(1- \delta(R)) \int_{\R^N}|\Delta (\varphi u)|^2 + |\nabla (\varphi u)|^2 \, dx
\leq C(R) +\left( \frac{4}{N^2}\left(6+ 8 \eps + 2\eps^2\right) + {\tau}\right)\int_{\R^N}|\Delta (\varphi u)|^2  \, dx\\
&+\left(\dfrac{2}{N-2}\right)^2\left( \frac{1}{4R^2}\left(10 + \frac{1}{\tau}+ \frac {6}{\eps} + 2 \eps \right)+ \left( 1 + \lambda^2_2 + \frac{2 \lambda_2^2}{\eps}\right) + \frac{\lambda_2}{R} \right) \int_{\R^N} |\nabla(\varphi u)|^2 \, dx.
\end{split}
\end{align}
Since $N \geq 5$ and since $\delta(R) \to 0$ as $R \to \infty$, we can choose $\eps, \tau>0$ small enough, and $R>0$ large enough so that
\begin{align*}
\left( \frac{4}{N^2}\left(6+ 8 \eps + 2\eps^2\right) + {\tau}\right) < 1- \delta,
\end{align*}
and
\begin{align*}
\left(\dfrac{2}{N-2}\right)^2\left( \frac{1}{4R^2}\left(10 + \frac{1}{\tau}+ \frac {6}{\eps} + 2 \eps \right)+ \left( 1 + \lambda^2_2 + \frac{2 \lambda_2^2}{\eps}\right) + \frac{\lambda_2}{R} \right)<1- \delta.
\end{align*}
We just show that there exists a constant $C>0$ depending on $R$ only such that 
$$\int_{\R^N}|\nabla (\varphi u)|^2 \, dx \leq C.$$ 
It now follows from \eqref{hardy} that 
$$\int_{|x| \geq 2R} \frac{|\varphi u|^2}{|x|^2}\, dx
\leq C$$
uniformly with respect to $R_1$. Finally, we take $R_1 \to\infty$ and we observe that
$$
\dfrac{|\varphi u|^2}{|x|^2}  \to u^2  \quad \mbox{ a.e for } \quad |x| \geq 2R.$$
Fatou's Lemma then implies that $u \in L^2(\R^N \backslash B_{2R}(0))$ so that $ u \in L^2(\R^N)$.
\end{proof}

\end{document}